\documentclass[review]{elsarticle}
\usepackage{geometry}
\usepackage{lineno,hyperref}
\usepackage{dsfont}
\usepackage{tikz}
\usepackage{amsmath,amssymb,amsfonts,amsthm}

\usepackage{algorithm}
\usepackage{bm}
\usepackage{bookmark}
\usepackage{bm}
\modulolinenumbers[5]
\usepackage[all]{xy}

\newtheorem{thm}{Theorem}[section]
\newtheorem{lem}[thm]{Lemma}
\newtheorem{con}[thm]{Construction}
\newtheorem{proposition}[thm]{Proposition}
\newtheorem{cor}[thm]{Corollary}
\theoremstyle{definition}
\newtheorem{dfn}[thm]{Definition}
\newdefinition{rmk}[thm]{Remark}
\newtheorem{exmp}[thm]{Example}
\theoremstyle{remark}
\numberwithin{equation}{section}

\journal{Journal of \LaTeX\ Templates}

\makeatletter
\def\ps@pprintTitle{%
   \let\@oddhead\@empty
   \let\@evenhead\@empty
   \def\@oddfoot{\reset@font\hfil\thepage\hfil}
   \let\@evenfoot\@oddfoot
}
\makeatother

\begin{document}

\begin{frontmatter}

\title{The application of representation theory in directed strongly regular graphs}
\author[]{Yiqin He\corref{cor1}}
\ead{2014750113@smail.\;xtu.\;edu.\;cn}
\author[]{Bicheng Zhang\corref{cor1}}
\ead{zhangbicheng@xtu.\;edu.\;cn}

\address{School of Mathematics and Computational Science, Xiangtan Univerisity, Xiangtan, Hunan, 411105, PR China}
\cortext[cor1]{Corresponding author}



%

\begin{abstract}
The concept of directed strongly regular graphs (DSRG) was introduced by Duval in 1988 \cite{A}.\;In the present paper,\;we use representation theory of finite groups in order to investigate the directed strongly regular Cayley graphs.\;We first show that a Cayley graph $\mathcal{C}(G,S)$ is not a directed strongly regular graph if $S$ is a union of some conjugate classes of $G$.\;This generalizes an earlier result of Leif K.\;J{\o}rgensen \cite{J1} on abelian groups.\;Secondly,\;by using induced representations,\;we have a look at the Cayley graph $\mathcal{C}(N\rtimes_\theta H, N_1\times H_1)$ with $N_1\subseteq N$ and $H_1\subseteq H$,\;determining its characteristic polynomial and its minimal polynomial.\;Based on this result,\;we generalize the semidirect product method of Art M. Duval and Dmitri Iourinski in \cite{D} and obtain a larger family of directed strongly regular graphs.\;Finally,\;we construct some directed strongly regular Cayley graphs on dihedral groups,\;which partially generalize the earlier results of Mikhail Klin,\;Akihiro Munemasa,\;Mikhail Muzychuk,\;and Paul Hermann Zieschang in \cite{K1}.\;By using character theory,\;we also give the characterization of directed strongly regular Cayley graphs $\mathcal{C}(D_n,X\cup Xa)$ with $X\cap X^{(-1)}=\emptyset$.\;

\end{abstract}
\end{frontmatter}
\section{Introduction}

A \emph{directed strongly regular graph} (DSRG,\;since it will appear many times, it will be abbreviated as "DSRG" in the following) with parameters $( n, k, \mu ,\lambda , t)$ is a $k$-regular directed graph on $n$ vertices such that every vertex is on $t$ 2-cycles,\;and the number of paths of length two from a vertex $x$ to a vertex $y$ is $\lambda$ if there is an edge directed from $x$ to $y$ and it is $\mu$ otherwise.\;A DSRG with $t=k$ is an (undirected) strongly regular graph (SRG).\;Duval showed that a DSRG with $t=0$ is a doubly regular tournament.\;Therefore,\;it is usually assumed that $0<t<k$.\;

Another definition of a directed strongly regular graph can be given in terms of adjacency matrices.\;Let $D$ be a directed graph
with $n$ vertices.\;Let $A=\mathbf{A}(D)$ denote the adjacency matrix of $D$;\;$I = I_n$ denote the $n\times n$ identity matrix;\;$J = J_n$ denote the all-ones matrix.\;Then $D$ is a DSRG with parameters $(n,k,\mu,\lambda,t)$ if
and only if (i) $JA = AJ = kJ$ and (ii) $A^2=tI+\lambda A+\mu(J-I-A)$.\;Duval \cite{A} developed necessary conditions on the parameters of $(n,k,\mu,\lambda,t)$-DSRG and calculated the spectrum of a DSRG.\;
\begin{proposition}(\cite{A})\label{p-1}
A DSRG with parameters $( n,k,\mu ,\lambda ,t)$ with $0<t<k$ satisfy
\begin{equation}\label{1.1}
k(k+(\mu-\lambda ))=t+\left(n-1\right)\mu,
\end{equation}
\begin{equation*}
{{d}^{2}}={{\left( \mu -\lambda  \right)}^{2}}+4\left(t-\mu\right)\text{,}d|2k-(\lambda-\mu)(n-1),
\end{equation*}
\begin{equation*}
\frac{2k-(\lambda-\mu)(n-1)}{d}\equiv n-1(\mathrm{mod}\;2)\text{,}\left|\frac{2k-(\lambda-\mu)(n-1)}{d}\right|\leqslant n-1,
\end{equation*}
where $d$ is a positive integer, and
\begin{equation}\label{1.2}
0\leqslant \lambda <t,0<\mu \leqslant t,-2\left( k-t-1 \right)\leqslant \mu -\lambda \leqslant 2\left( k-t \right).
\end{equation}
\end{proposition}
\begin{proposition}(\cite{A})\label{p-2}
A DSRG with parameters $( n, k, \mu ,\lambda , t)$ has three distinct integer eigenvalues
\begin{equation}\label{1.3}
k>\rho =\frac{1}{2}\left( -\left( \mu -\lambda  \right)+d \right)>\sigma =\frac{1}{2}\left( -\left( \mu -\lambda  \right)-d \right).\;
\end{equation}
The multiplicities are
\begin{equation}\label{1.4}
1,\;m_\rho=-\frac{k+\sigma \left( n-1 \right)}{\rho -\sigma },\;m_\sigma=\frac{k+\rho \left( n-1 \right)}{\rho -\sigma }
\end{equation}
respectively.\;
\end{proposition}
\begin{proposition}(\cite{A})\label{p-3}
If $G$ is a DSRG with parameters $( n, k, \mu ,\lambda , t)$,\;then its complement $G'$ is also a DSRG with parameters $(n',k',\mu',\lambda' ,t')$,\;where $k'=(n-2k)+(k-1)$,\;$\lambda'=(n-2k)+(\mu-2)$,\;$t'=(n-2k)+(t-1)$ and $\mu'=(n-2k)+\lambda$.\;
\end{proposition}
\begin{rmk}\label{mini}Let $D$ be a DSRG with parameters $(n,k,\mu,\lambda,t)$,\;and let $A$ be the adjacency matrix of $D$.\;Then $A$ has minimal polynomial of degree $3$,\;that is
 $(A-kI_n)(A^2+(\mu-\lambda)A+(\mu-t)I_n)=(A-kI_n)(A-\rho I_n)(A-\sigma I_n)=0$.\;
\end{rmk}

We introduce some notations of multisets.\;Let $A$ be a multiset together with a \emph{multiplicity function}
$\Delta_{A}:A\rightarrow \mathbb{N}$,\;where $\Delta_{A}(a)$ counting ``how many times of $a$ occurs in the multiset $A$".\;We say $x$ belongs to $A$\;(i.e.\;$x\in A$)\;if $\Delta_A(x)>0$.\;In the following, $A$ and $B$ are multisets,\;with multiplicity functions $\Delta_{A}$ and $\Delta_{B}$.\;

\begin{itemize}
  \item \textbf{Union,\;$A\uplus B$}:\;the union of multisets $A$ and $B$,\;is defined by $\Delta_{A\uplus B}=\Delta_A+\Delta_B$;
  \item \textbf{Scalar multiplication,\;$n\oplus A$}:\;the scalar multiplication of a multiset $A$ by a natural number $n$ by,\;is defined by $\Delta_{n\oplus A}=n\Delta_A$.
\end{itemize}
If $A$ and $B$ are usual sets,\;we use $A\cup B$,\;$A\cap B$ and  $A\setminus B$ denote the usual union,\;intersection and difference of $A$ and $B$.\;For example,\;if $A=\{1,2\}$ and $B=\{1,3\}$,\;then $A\uplus B=\{1,1,2,3\}$,\;$A\cup B=\{1,2,3\}$ and $A\setminus B=\{2\}$.\;

We will always use ``set'' to mean the usual set (i.e.,\;not multiset) unless otherwise indicated.

Let $G$ be a group and $e_G$ or $e$ be the identity element of $G$.\;For any subset $S$ of a multiplicative group $G$ and an integer $t$,\;we use $S^{(t)}$ to denote the set $\{a^{t}|a\in S\}$.\;A subset $S$ is called antisymmetric if $S\cap S^{(-1)}=\emptyset$.\;

Leif K.\;J{\o}rgensen in \cite{J1} showed that a DSRG cannot be a Cayley graph of an abelian group.\;The goal of Section 3 is to generalize this result.\;By using character theory of finite groups,\;we show that a Cayley graph $\mathcal{C}(G,S)$ is not a DSRG if $S$ is a union of some conjugate classes of $G$.\;To prove this theorem,\;we calculate the spectra of such Cayley graphs $\mathcal{C}(G,S)$ in terms of the irreducible characters of $G$,\;we have the following theorem.\;
\begin{thm}Let $\mathbf{Irr}(G)=\{\chi_1,\cdots,\chi_r\}$ be the set of irreducible characters of $G$,\;then the eigenvalues of Cayley graph $\mathcal{C}(G,S)$ with $\overline{S}\in \bm{\mathcal{Z}}(\mathbb{Z}[G])$ are $\lambda_1,\cdots,\lambda_r$ of multiplicities $n_1^2,n_2^2,\cdots,n_r^2$ respectively,\;where
\begin{equation*}
\lambda_i=\frac{1}{n_i}\sum_{s\in S}\chi_i(s)
\end{equation*}
for $1\leqslant i\leqslant r$,\;and $n_1,n_2,\cdots,n_r$ are degrees of $\chi_1,\cdots,\chi_r$ respectively.\;
\end{thm}
This lemma enables us to establish the following theorem.\;
\begin{thm}A DSRG cannot be a Cayley graph $\mathcal{C}(G,S)$ with $\overline{S}\in \bm{\mathcal{Z}}(\mathbb{Z}[G])$.\;
\end{thm}

In Section $4$,\;a more general theorem is concerned.\;We show that a Cayley graph $\mathcal{C}(G,S)$ is not a DSRG if $S=\mathbf{C}\cup T$,\;where $C$ is a union of some conjugate classes of $G$ and $T$ is a subset of $G\setminus \mathbf{C}$ such that $T\cap T^{(-1)}=\emptyset$.\;

 Art M. Duval and Dmitri Iourinski \cite{D} constructed DSRGs by using semidirect product.\;This method motivates us to investigate Cayley graphs $\mathcal{C}(N\rtimes_\theta H,S)$.\;In Section 5,\;by using induced representations,\;we give the characteristic polynomial and  minimal polynomial of $\mathcal{C}(N\rtimes_\theta H, N_1\times H)$ with $N_1\subseteq N$.\;Based on this result,\;we can generalize the semidirect product construction method of Art M.\;Duval and Dmitri Iourinski in \cite{D} and  obtain a larger family of DSRGs.\;We get the following construction.\;
\begin{con}Let $\mathcal{C}(N\rtimes_\theta H, N_1\times H)$ be a Cayley graph such that $$\overline{N_1^\flat}=x_1\overline{N}+x_2 e_N$$ for some integers $d>x_1>0$ and $x_2<0$,\;then the
Cayley graph $\mathcal{C}(N\rtimes_\theta H, N_1\times H)$ is a DSRG with parameters $$\left(md, x_1m+x_2, \frac{x_1(x_1m+x_2)}{d}, x_2+ \frac{x_1(x_1m+x_2)}{d},  \frac{x_1(x_1m+x_2)}{d}\right).$$
\end{con}
The notation $\overline{N_1^\flat}$ appearing in the above construction will be defined in (\ref{musical}).\;

Let $C_n=\langle x:x^n=e\rangle$ be a (multiplicative) cyclic group of order $n$,\;generated by $x$.\;The \emph{dihedral group} $D_n$ is the group of symmetries of a regular polygon,\;and it can be viewed as a semidirect product of two cyclic groups $C_n=\langle x\rangle$ of order $n$ and $C_2=\langle a\rangle$ of order $2$.\;The presentation of $D_n$ is $D_n=C_n\rtimes C_2=\langle x,a|x^n=a^2=e,ax=x^{-1}a\rangle$.\;The cyclic group $C_n$ is a normal subgroup of $D_n$.\;All subgroups of $C_n$ are of form $\langle x^v\rangle$,\;with $v$ as a positive divisor of $n$.\;Let $\langle x^v\rangle$ be a subgroup of $C_n$,\;then a transversal for $\langle x^v\rangle$ in $C_n$ is $\{e,x^1,\cdots,x^{v-1}\}$.\;Let $T$ be a subset of $\{e,x^1,\cdots,x^{v-1}\}$,\;we define
\[T\langle x^v\rangle=\bigcup_{x^a\in T}x^a\langle x^v\rangle,\]
where $x^a\langle x^v\rangle$ are coset of $\langle x^v\rangle$ in $C_n$,\;for $x^a\in T$.\;

In Section $6$,\;we construct some directed strongly regular Cayley graphs on dihedral groups.\;In the following constructions,\;$v$ is a positive divisor of $n$ and $l=\frac{n}{v}$.\;
\begin{con}
Let $v$ be an odd positive divisor of $n$.\;Let $T$ be a subset of $\{x^1,\cdots,x^{v-1}\}$,\;and $X$ be a subset of ${C}_n$ satisfy the following conditions:\\
$(i)$\;$X=T\langle x^v\rangle$.\\
$(ii)$\;$X\cup X^{(-1)}=C_{n}\setminus \langle x^v\rangle$.\\
Then the Cayley graph $\mathcal{C}(D_n,X\cup Xa)$ is a DSRG with parameters $\left(2n,n-l,\frac{n-l}{2},\frac{n-l}{2}-l,\frac{n-l}{2}\right)$.\;
\end{con}

\begin{con}Let $v>2$ be an even positive divisor of $n$.\;Let $T$ be a subset of $\{x^1,\cdots,x^{v-1}\}$,\\and $X$ be a subset of ${C}_n$ satisfy the following conditions:\\
$(i)$\;$X=T\langle x^v\rangle$.\\
$(ii)$\;$X\cup X^{(-1)}=({C}_{n}\setminus \langle x^v\rangle)\uplus(x^{\frac{v}{2}}\langle x^v\rangle)$.\\
$(iii)$\;$X\cup (x^{\frac{v}{2}}X)={C}_{n}$.\\
Then the Cayley graph $\mathcal{C}(D_n,X\cup Xa)$ is a DSRG with parameters $\left(2n,n,\frac{n}{2}+l,\frac{n}{2}-l,\frac{n}{2}+l\right)$.\;
\end{con}

\begin{con}Let $v$ be an odd positive divisor of $n$.\;Let $T$ be a subset of $ \{e,x^1,\cdots,x^{v-1}\}$ with $e\in T$,\;and $X,Y\subseteq{C}_n$ satisfy the following conditions:\\
$(i)$\;$Y=T\langle x^v\rangle=X\cup \langle x^v\rangle$.\\
$(ii)$\;$Y\cup Y^{(-1)}=C_{n}\uplus \langle x^v\rangle$.\\
Then the Cayley graph $\mathcal{C}(D_n,X\cup Ya)$ is a DSRG with parameters $\left(2n,n,\frac{n+l}{2},\frac{n-l}{2},\frac{n+l}{2}\right)$.\;
\end{con}

The above constructions partially generalize the earlier results of Mikhail Klin,\;Akihiro Munemasa,\;Mikhail Muzychuk,\;and Paul Hermann Zieschang in \cite{K1}.

We consider the problem of characterizing directed strongly regular Cayley graphs on dihedral groups.\;We give a characterization of the directed strongly regular Cayley graph $\mathcal{C}(D_n,X\cup Xa)$ with $X\cap X^{(-1)}=\emptyset$.\;We have been proved:
\begin{thm}A Cayley graph $\mathcal{C}(D_n,X\cup Xa)$ with $X\cap X^{(-1)}=\emptyset$ is a DSRG with parameters $\left(2n,2|X|,\mu,\lambda,t\right)$ if and only if there exists a subset $T$ of $\{x^1,\cdots,x^{v-1}\}$ satisfies the following conditions:\\
$(i)$\;$X=T\langle x^v\rangle$;\\
$(ii)$\;$X\cup X^{(-1)}=C_{n}\setminus \langle x^v\rangle$,
where $v=\frac{n}{\mu-\lambda}$.\;
\end{thm}

\section{Preliminaries}
We now give some basic definitions and notations of \emph{Cayley graph},\;\emph{Cayley multigraph},\;\emph{group ring},\;\emph{group algebra} and representation theory.\;

Let $\Gamma$ be a digraph and the adjacent matrix of $\Gamma$ is denoted by $\textbf{A}(\Gamma)$.\;For a matrix $A$,\;we denote the characteristic polynomial and the minimal polynomial of $A$ by $\mathcal{F}_A(x)$ and $\mathcal{M}_A(x)$.\;
\begin{dfn}\label{d-Cayleydigraph}(Cayley graph)
Let $G$ be a finite group and $S\subseteq G\backslash \{e_G\}$.\;The Cayley graph of $G$ generated
by $S$,\;denoted by $\mathcal{C}(G,S)$,\;is the digraph $\Gamma$ such that $V(\Gamma)=G$ and $x\rightarrow y$ if and only if
$yx^{-1}\in S$,\;for any $x,y\in G$.
\end{dfn}

A multigraph version of Cayley graph is Cayley multigraph.\;
\begin{dfn}\label{d-Cayleymultidigraph}(Cayley multigraph)
Let $G$ be a finite group and $S$ be a multisubset of $G\backslash \{e_G\}$.\;The Cayley multigraph of $G$ generated
by $S$,\;denoted by $\mathcal{C}(G,S)$,\;is the digraph $\Gamma$ such that $V(\Gamma)=G$ and the number of acrs from $x$ to $y$ is $\Delta_{S}(y^{-1}x)$.
\end{dfn}

\begin{dfn}\label{d-GroupRing}(Group Ring)
For any group $G$ and ring $R$,\;the group ring $R[G]$ of $G$ over $R$,\;denoted by $R[G]$,\;consists of all finite formal
sums of elements of $G$,\;with coefficients from $R$.\;i.e.,
\begin{equation*}
R[G]=\left\{\sum_{g\in G}r_gg:r_g\in R,\; r_g\neq 0\text{ for finite g}\right\}.\;
\end{equation*}
The operations $+$ and $\cdot$ on $R[G]$ are given
by
\begin{equation*}
\sum_{g\in G}r_gg+\sum_{g\in G}s_gg=\sum_{g\in G}(r_g+s_g)g,
\end{equation*}
\begin{equation*}
\left(\sum_{g\in G}r_gg\right)\cdot\left(\sum_{g\in G}r_gg\right)=\left(\sum_{g\in G}t_gg\right),\;t_g=\sum_{g'g''=g}r_{g'}s_{g''}.
\end{equation*}
\end{dfn}

The \emph{center} of $R[G]$ is the set of elements of $R[G]$ which commute with all the elements in $R[G]$.\;We denote the center of $R[G]$ by $\bm{\mathcal{Z}}(R[G])$.\;

For any multisubset $X$ of $G$,\;let $\overline{X}$ denote the element of the group ring $R[G]$ that is the sum of all elements of $X$,\;i.e.,
\begin{equation*}
\overline{X}=\sum_{x\in X}\Delta_X(x)x.
\end{equation*}

The lemma below allows us to express a sufficient and necessary condition for a Cayley graph to be directed strongly regular in terms of group ring.
\begin{lem}\label{l-DSRGCayleyGraphsGruopRing}
A Cayley graph $\mathcal{C}(G,S)$  is a DSRG with parameters $(n, k, \mu,  \lambda, t)$ if and only if $|G| = n$, $|S|= k$, and
\[\overline{S}^2=te+\lambda\overline{S}+\mu(\overline{G}-e-\overline{S}).\]
\end{lem}
\subsection{Representation theory}
This section is based on the book J.\;L.\;Alperin and Rowen B.\;Bell,\;\emph{Groups and representations}\cite{JL}.\;

Let $F$ be a field  and $G$ be a finite group,\;then the group ring $F[G]$ is not only a ring but also an $F$-vector space having $G$ as a basis and hence having finite dimension $|G|$.\;In this case,\;$F[G]$ becomes an algebra which we call it \emph{group algebra}.\

Let $V$ be a finite-dimensional vector space over the field $F$.\;We define the general linear group $GL(V)$ to be the group of all invertible linear transformations of $V$.\;We use notation $\mathbf{1}_V$ to denote the \emph{identity linear transformation} of $V$.\;We define the \emph{general linear group} $GL_{n}(F)$ to be the set consisting of all invertible matrices over $F$ of order $n$.

An (finite-dimensional) $F$-\emph{linear representation} $(V,\rho)$ of $G$ in $V$ is a homomorphism $\rho:G\rightarrow GL(V)$.\;If $V$ is a finite-dimensional vector space,\;the degree of $(V,\rho)$ is the dimension of $V$ and we denote it by $\deg\rho$.\;We also say that $V$ is a \emph{representation space} of $G$ with respect to the representation $\rho$.\;

Let $(V,\rho)$ be an $F$-linear representation.\;A subspace $W$ of $V$ is $G$-\emph{invariant},\;if for all $u\in W$ and  $g\in G$,\;we have $\rho(g)(u)\in W$.\;If $W$  is  a $G$-\emph{invariant subspace},\;we can restrict $\rho$ on $W$ and obtain a representation $\rho|W:G\rightarrow GL(W)$.\;We say $\rho|W$ is a \emph{subrepresentation} of $\rho$.\;An $F$-linear representation $(V,\rho)$ is \emph{irreducible} if $V$ doesn't have any untrivial subrepresentation.\;

Let $(V,\rho)$ be a representation of  group $G$ with $\dim_{F}V=n$,\;let $B=(e_i)$ be a basis of $V$,\;and let $\rho_B(g)$ be the matrix of $\rho(g)$ with respect to the basis $B$ for any $g\in G$.

Let $\rho$ and $\rho'$ be two representations of the same group $G$ in vector spaces $V$ and $V'$ respectively,\;these representations are said to be \emph{equivalent} or \emph{isomorphic} if there exists a linear isomorphism $\omega:V\rightarrow V'$ such that the diagram
\begin{equation*}
\xymatrix{
V \ar[d]_\omega \ar[r]^{\rho(g)} & V \ar[d]^\omega       \\
   V' \ar[r]_{\rho(g)}   & V'
}
\end{equation*}
commutes for each $g\in G$.\;

Let $\overline{\mathbf{Irr}}_F(G)$ be the set of all the non-equivalent irreducible $F$-representations of $G$.\;

The \emph{left regular representation} of $G$ is the homomorphism $\rho_{reg}:G\rightarrow GL(\mathbb{C}[G])$ defined by
\[\rho_{reg}(g)\left(\sum_{h\in G}r_hh\right)=\sum_{h\in G}r_h(gh)\]
for $g\in G$.\;

\subsection{Character theory of finite groups}
Let $F=\mathbb{C}$ and $(V,\rho)$ be a $\mathbb{C}$-linear representation of $G$,\;the $\mathbb{C}$-character $\chi_{\rho}:G\rightarrow \mathbb{C}$ of $\rho$ is defined by
\[\chi_{\rho}(g)=\mathrm{tr}\rho(g),\]
where $\mathrm{tr}$ is the \emph{trace function}.\;The kernel of $\chi_{\rho}$ is defined by $\mathcal{K}_{\chi_\rho}=\{g\in G|\chi_\rho(g)=\chi_\rho(1)\}$,\;which is a normal subgroup of $G$.\;

The character of an irreducible representation is called an \emph{irreducible  character}.\;Let $\mathbf{Irr}(G)$ be the set of all the irreducible $\mathbb{C}$-characters of $G$.\;

Let $G$ be a finite group with class number $r$,\;i.e.,\;the number of distinct conjugacy classes of $G$.\;Let $C_1,\;C_2,\;\cdots,\;C_r$ be the $r$ conjugacy classes of $G$,\;and $g_1,\;g_2,\;\cdots,\;g_r$ be the representatives of the $r$ conjugacy classes respectively.\;The following lemmas give some elementary theorems in character theory,\;which can be referred to \cite{JL} and \cite{JP}.

\begin{lem}\label{l-irreduciblecharacters}Let $G$ be a finite group.\;Then we have the following theorems:\\
(1)The number of all the non-equivalent irreducible $\mathbb{C}$-representations (characters) of $G$ is equal to $r$,\;the class number of $G$.\;Equivalently,\;$$|\overline{\mathbf{Irr}}_{\mathbb{C}}(G)|=|\mathbf{Irr}(G)|=r.\;$$
(2)For any $\mathbb{C}$-linear representation $(V,\rho)\in\overline{\mathbf{Irr}}_{\mathbb{C}}(G)$,\;$\chi_{\rho}(1)=\deg\rho=\dim V$.\;\\
(3)(First orthogonality relation) Let $\chi_1,\cdots,\chi_r$ be all the $r$ irreducible $\mathbb{C}$-characters of $G$,\;then
\[\frac{1}{|G|}\sum_{h=1}^r\chi_i(g_h)\overline{\chi_j(g_h)}|C_h|=\delta_{ij}\]
for any $1\leqslant i,j\leqslant r$.\\
(4)(Second orthogonality relation)Let $g_1,g_2,\cdots,g_r$ be the conjugacy class representatives of $G$,\;then
\[\frac{1}{|G|}\sum_{h=1}^r\chi_h(g_i)\overline{\chi_h(g_j)}|C_i|=\delta_{ij}\]
for any $1\leqslant i,j\leqslant r$.\;\\
(5)(Character table) The character table of $G$ is a $r\times r$ matrix $\mathcal{X}=(\chi_i(g_j))_{r\times r}$.\;Let $D_1=\mathrm{diag}\{|C_1|,\;\cdots,\;|C_r|\}$ and $D_2=\mathrm{diag}\{\chi_1(1),\cdots,\chi_r(1)\}\overset{\text{\textmd{def}}}=\mathrm{diag}\{n_1,\cdots,n_r\}$.\;Then the orthogonality relations can be rewritten by the matrices form:
\begin{equation}\label{CharacterTable}
\mathcal{X}D_1\overline{\mathcal{X}}^{T}=|G|I_r,\;\overline{\mathcal{X}}^{T}\mathcal{X}=|G|D_1^{-1}.
\end{equation}
\end{lem}

The inversion formula is given now.\;
\begin{lem}\label{l-Inversionformula}(Inversion formula) Let $G$ be a finite group and $A=\sum\limits_{g\in G}a_gg\in \mathbb{C}[G]$,\;then
\[a_g=\frac{1}{|G|}\sum_{\chi\in\mathbf{ Irr}(G)}\chi(Ag^{-1})\chi(e),\forall g\in G.\]
In particular,\;if $G$ is an abelian group,\;then the above formula becomes
\[a_g=\frac{1}{|G|}\sum_{\chi\in\mathbf{ Irr}(G)}\chi(A)\overline{\chi(g)},\forall g\in G.\]
\end{lem}

\begin{cor}Let $G$ be an abelian group and $A,B\in \mathbb{C}[G]$,\;then $A=B$ if and only if
$\chi(A)=\chi(B)$  for all $\chi\in \mathbf{ Irr}(G)$.\;
\end{cor}

\section{A DSRG cannot be a Cayley graph \texorpdfstring{$\mathcal{C}(G,S)$ with $\overline{S}\in \bm{\mathcal{Z}}(\mathbb{Z}[G])$}{Lg}}
 L.\;K.\;J{\o}rgensen \cite{J1} proved that a DSRG cannot be a Cayley graph $\mathcal{C}(G,S)$ of an abelian group.\;This section gives a generalization of this result.\;We show that a DSRG cannot be a Cayley graph $\mathcal{C}(G,S)$,\;where $S$ is a union of some conjugate classes of $G$.\;

At first,\;we give a lemma about the group algebra $\mathbb{C}[G]$ and the center $\bm{\mathcal{Z}}(\mathbb{C}[G])$.\;Let $e$ be the identity element of $G$.\;It is clear that $e$ is the unity of the group algebra $\mathbb{C}[G]$.\;
\begin{lem}\label{l-Structure}Let $G$ be a finite group and $\mathbf{Irr}(G)=\{\chi_1,\cdots,\chi_r\}$ be the set of all the irreducible $\mathbb{C}$-characters of $G$.\;Define
\[e_i=\frac{\chi_i(1)}{|G|}\sum_{g\in G}\overline{\chi_i{(g)}}g\]
for $1\leqslant i\leqslant r$.\;Then\\
(1)The $r$ elements $e_1,e_2,\cdots,e_r$ form a complete family of primitive central idempotents with $e_1+e_2+\cdots+e_r=e$.\;The group algebra $\mathbb{C}[G]$ can be decomposited into the direct sum of minimal two-sided ideals (also $(\mathbb{C}[G],\mathbb{C}[G])$-bimodules),\;that is
\[\mathbb{C}[G]=\mathbb{C}[G]e_1\oplus\mathbb{C}[G]e_2\cdots\oplus\mathbb{C}[G]e_r.\]
(2)The $r$ elements $\{e_1,e_2,\cdots,e_r\}$ form a $\mathbb{C}$-basis of $\bm{\mathcal{Z}}(\mathbb{C}[G])$.\;Then,\;as a vector space,\;we have
\[\bm{\mathcal{Z}}(\mathbb{C}[G])=\mathbb{C}e_1\oplus\mathbb{C}e_2\cdots\oplus\mathbb{C}e_r.\]
\end{lem}

The \emph{spectrum} of a graph $\Gamma$ is the set of eigenvalues of $\mathbf{A}(\Gamma)$,\;i.e.,\;the adjacency matrix of $\Gamma$,\;together with their multiplicities.\;The spectrum of a graph is an important algebraic invariant.\;In general,\;we still cannot obtain a simple and explicit formula of eigenvalues of a Cayley graph $\mathcal{C}(G,S)$.\;We now list some known  results about the spectrum of Cayley graph $\mathcal{C}(G,S)$.\;

When $G$ is an abelian group,\;we have:
\begin{thm}({\cite{ST}},\;Theorem 5.4.10) Let $G$ be an abelian group,\;$S$ be a subset of $G\setminus\{e\}$,\;and $\mathbf{Irr}(G)=\{\chi_1,\chi_2,\cdots,\chi_n\}$ be the set of irreducible characters of $G$.\;Then the eigenvalues of the adjacency matrix  $\mathbf{A}(\mathcal{C}(G,S))$ are
\begin{equation*}
\lambda_i=\sum_{s\in S}\chi_i(s)
\end{equation*}
for any $1\leqslant i\leqslant n$.\;
\end{thm}

In \cite{L},\;L.\;Babai derived an expression for the spectrum of the Cayley graph $\mathcal{C}(G,S)$ in terms of irreducible characters of the group $G$.\;Let $\mathbf{Irr}(G)=\{\chi_1,\chi_2,\cdots,\chi_r\}$ be the set of irreducible characters of $G$,\;and
$n_1,n_2,\cdots,n_r$ are degrees of $\chi_1,\cdots,\chi_r$ respectively.\;
\begin{thm}(\cite{L})Let $\mathcal{C}(G,S)$ be a Cayley graph.\;Then
\begin{equation*}
\lambda_{i,1}^t+\lambda_{i,2}^t+\cdots+\lambda_{i,n_i}^t=\sum_{g_1,\cdots,g_t\in S}\chi_i(g_1\cdots g_t)
\end{equation*}
for any positive integer $t$ and $1\leqslant i\leqslant r$,\;where $\lambda_{i,1},\lambda_{i,2},\cdots,\lambda_{i,n_i}$ are eigenvalues of $\mathcal{C}(G,S)$,\;of the same multiplicities $n_i$.\;
\end{thm}

\subsection{The spectra of Cayley graphs \texorpdfstring{$\mathcal{C}(G,S)$ with $\overline{S}\in \bm{\mathcal{Z}}(\mathbb{Z}[G])$}{Lg}}
Let $\rho_{reg}$ be the left regular representation of $G$.\;If $G=\{g_1,g_2,\cdots,g_{|G|}\}$ were
chosen as a basis of the group algebra $\mathbb{C}[G]$,\;then the matrix $\rho_{reg,G}(\overline{S})$ is an adjacent matrix of Cayley graph $\mathcal{C}(G,S)$,\;for any $S\subseteq G\setminus\{e_G\}$.

Let $C_1,C_2,\cdots,C_r$ be the conjugacy classes of the group $G$,\;then $\overline{C_1},\overline{C_2},\cdots,\overline{C_r}$ are called the \emph{conjugacy class sum} of $G$.\;It is clear that $\{\overline{C_1},\overline{C_2},\cdots,\overline{C_r}\}$ is a $\mathbb{C}$-basis of $\bm{\mathcal{Z}}(\mathbb{C}[G])$,\;i.e.,\;
\begin{equation*}
\bm{\mathcal{Z}}(\mathbb{C}[G])=\mathbb{C}\overline{C_1}\oplus\mathbb{C}\overline{C_2}\cdots\oplus\mathbb{C}\overline{C_r}.
\end{equation*}
This shows that a subset $S$ is a union of some conjugate classes of $G$ if and only if $\overline{S}\in \bm{\mathcal{Z}}(\mathbb{C}[G])\cap \mathbb{Z}[G]=\bm{\mathcal{Z}}(\mathbb{Z}[G])$.\;From Lemma \ref{l-Structure},\;$\{e_1,e_2,\cdots,e_r\}$ is another basis of $\bm{\mathcal{Z}}(\mathbb{C}[G])$,\;i.e.,\;
\begin{equation*}
\bm{\mathcal{Z}}(\mathbb{C}[G])=\mathbb{C}{e_1}\oplus\mathbb{C}{e_2}\cdots\oplus\mathbb{C}{e_r}
\end{equation*}
where
\begin{equation}\label{5.1.1}
e_j=\frac{\chi_j(1)}{|G|}\sum_{g\in G}\overline{\chi_j{(g)}}g
\end{equation}
for $1\leqslant j\leqslant r$.\;The following lemma gives the spectrum of the Cayley graph $\mathcal{C}(G,S)$,\;where $S$ is a union of some conjugate classes of $G$.\;
\begin{thm}\label{l-eigenvalues}Let $\mathcal{C}(G,S)$ be a Cayley graph with $\overline{S}\in \bm{\mathcal{Z}}(\mathbb{Z}[G])$ and  $\mathbf{Irr}(G)=\{\chi_1,\cdots,\chi_r\}$ be the set of irreducible characters of $G$.\;Then the eigenvalues of such Cayley graph are $\lambda_1,\cdots,\lambda_r$ of multiplicities $n_1^2,n_2^2,\cdots,n_r^2$ respectively,\;where
\begin{equation*}
\lambda_i=\frac{1}{n_i}\sum_{s\in S}\chi_i(s)
\end{equation*}
for $1\leqslant i\leqslant r$,\;and $n_1,n_2,\cdots,n_r$ are degrees of $\chi_1,\cdots,\chi_r$ respectively.\;
\end{thm}
\begin{proof}From Lemma \ref{l-Structure}(2),\;assuming $$\overline{S}=\sum\limits_{j=1}^r\lambda_je_j$$ for some complex numbers $\lambda_1,\lambda_2,\cdots,\lambda_r$.\;Therefore
\[\rho_{reg}(\overline{S})=\sum_{j=1}^r\lambda_j \rho_{reg}(e_j).\]
Then we have
\begin{equation*}
\begin{aligned}
\overline{S}&=\sum_{j=1}^r\lambda_je_j=\frac{1}{|G|}\sum_{j=1}^r\lambda_jn_j\sum_{g\in G}\overline{\chi_j{(g)}}g\\
&=\sum_{g\in G}\left(\frac{1}{|G|}\sum_{j=1}^r\lambda_jn_j\overline{\chi_j{(g)}}\right)g=\sum_{g\in G}\Delta_S(g)g
\end{aligned}
\end{equation*}
from the equation (\ref{5.1.1}).\;This gives that $(\lambda_1,\;\cdots,\;\lambda_r)^T$ satisfies the linear equation
\begin{equation*}
\overline{\mathcal{X}}^{T}D_2\left(\begin{array}{c}
\lambda_1\\
\lambda_2\\
\vdots\\
\lambda_r\\
\end{array}\right)=|G|\left(\begin{array}{c}
\Delta_S(g_1)\\
\Delta_S(g_2)\\
\vdots\\
\Delta_S(g_r)\\
\end{array}\right),
\end{equation*}
where $g_1,\;g_2,\;\cdots,\;g_r$ are the representatives of the $r$ conjugacy classes.\;Multiplying $\mathcal{X}D_1$ on both sides of the above equation,\;we can get
\begin{equation*}
\mathcal{X}D_1\overline{\mathcal{X}}^{T}D_2\left(\begin{array}{c}
\lambda_1\\
\lambda_2\\
\vdots\\
\lambda_r\\
\end{array}\right)=|G|D_2\left(\begin{array}{c}
\lambda_1\\
\lambda_2\\
\vdots\\
\lambda_r\\
\end{array}\right)=|G|\mathcal{X}D_1\left(\begin{array}{c}
\Delta_S(g_1)\\
\Delta_S(g_2)\\
\vdots\\
\Delta_S(g_r)\\
\end{array}\right)
\end{equation*}
by Lemma \ref{l-irreduciblecharacters}\;(5).\;Then
\begin{equation*}
\left(\begin{array}{c}
\lambda_1\\
\lambda_2\\
\vdots\\
\lambda_r\\
\end{array}\right)=D_2^{-1}\mathcal{X}D_1\left(\begin{array}{c}
\Delta_S(g_1)\\
\Delta_S(g_2)\\
\vdots\\
\Delta_S(g_r)\\
\end{array}\right).
\end{equation*}
Hence,\;we have
\begin{equation*}
\lambda_i=\frac{1}{n_i}\sum_{j=1}^r|C_j|\Delta_S(g_j)\chi_i(g_j)=\frac{1}{n_i}\sum_{s\in S}\chi_i(s)
\end{equation*}
for each $1\leqslant i\leqslant r$.\;It follows from Lemma \ref{l-Structure} that
\begin{equation*}
\mathbb{C}[G]=\mathbb{C}[G]e_1\oplus\mathbb{C}[G]e_2\cdots\oplus\mathbb{C}[G]e_r
\end{equation*}
is a decomposition of the group algebra into minimal two-sided ideals (also $(\mathbb{C}[G],\mathbb{C}[G])$-bimodules).\;For each $1\leqslant i\leqslant r$,\;the $\rho_{reg}(\overline{S})$ restricts to the submodule $\mathbb{C}[G]e_i$ is
\begin{equation*}
\begin{aligned}
\rho_{reg}(\overline{S})\Big|_{\mathbb{C}[G]e_i}&=\sum_{j=1}^r\left(\lambda_j \rho_{reg}(e_j)\Big|_{\mathbb{C}[G]e_i}\right)\\
&=\lambda_i\rho_{reg}(e_i)\Big|_{\mathbb{C}[G]e_i}\\
&=\lambda_i\mathbf{1}_{\mathbb{C}[G]e_i}\Big|_{\mathbb{C}[G]e_i},
\end{aligned}
\end{equation*}
this is a scalar multiplication.\;Thus the eigenvalues of $\rho_{reg}(\overline{S})\Big|_{\mathbb{C}[G]e_i}$ are $\lambda_i$ with multiplicity $n_i^2=\dim\mathbb{C}[G]e_i$.\;Then the result follows.\;
\end{proof}

\subsection{The proof of the main theorem}
\begin{lem}(see \cite{K1},\;Lemma 3.\;2)\label{l-withoutundirected}
Let $H$ be a regular non-empty directed graph without undirected edges,\;then $A=\mathbf{A}(H)$ has at least one non-real eigenvalue.\;
\end{lem}
\begin{thm}\label{t-conjugacy}A DSRG cannot be a Cayley graph $\mathcal{C}(G,S)$ with $\overline{S}\in \bm{\mathcal{Z}}(\mathbb{Z}[G])$.\;
\end{thm}
\begin{proof}Note that if $C$ is a conjugacy class of $G$,\;then so do $C^{(-1)}$.\;Define
$$S_1=\bigcup_{\{C_i:C_i^{(-1)}\subseteq S\}}C_i\;\text{and\;} S_2=\bigcup_{\{C_i:C_i^{(-1)}\nsubseteq S\}}C_i,\;$$
then $S_1$ and $S_2$ are also a union of some conjugacy classes of $G$,\;and hence $\overline{S_1},\overline{S_2}\in\bm{\mathcal{Z}}(\mathbb{Z}[G])$.\;Thus from Theorem \ref{l-eigenvalues},\;the eigenvalues of Cayley graphs $\mathcal{C}(G,S_1)$ and $\mathcal{C}(G,S_2)$ are $\{\lambda_{1,1},\;\cdots,\;\lambda_{1,r}\}$ and $\{\lambda_{2,1},\cdots,\lambda_{2,r}\}$ of the same multiplicities $n_1^2,n_2^2,\cdots,n_r^2$ respectively,\;where
$$\lambda_{1,i}=\frac{1}{n_i}\sum_{s\in S_1}\chi_i(s),\;\lambda_{2,i}=\frac{1}{n_i}\sum_{s\in S_2}\chi_i(s),\;$$
for $1\leqslant i\leqslant r$.\;Suppose Cayley graph $\mathcal{C}(G,S)$ is a DSRG,\;then all the eigenvalues of $\mathcal{C}(G,S)$ are real numbers.\;Note that $S_1\cup S_2=S$,\;then from Theorem \ref{l-eigenvalues},\;we have
\begin{equation*}
\lambda_i=\frac{1}{n_i}\sum_{s\in S}\chi_i(s)=\frac{1}{n_i}\sum_{s\in S_1}\chi_i(s)+\frac{1}{n_i}\sum_{s\in S_2}\chi_i(s)=\lambda_{1,i}+\lambda_{2,i}
\end{equation*}
and
\begin{equation*}
\begin{aligned}
\lambda_{1,i}&=\frac{1}{n_i}\sum_{s\in C_i,C_i\subseteq S_1}\chi_i(s)=\frac{1}{2n_i}\sum_{C_i\subseteq A}\left(\chi_i(\overline{C_i})+\chi_i(\overline{C_i^{(-1)}})\right)\\
&=\frac{1}{n_i}\sum_{C_i\subseteq A}\mathbf{R}(\chi_i(\overline{C_i}))\in\mathbb{R}
\end{aligned}
\end{equation*}
for any $1\leqslant i\leqslant r$,\;where $\mathbf{R}(z)$ denotes the real component of $z$.\;Therefore all the eigenvalues $\{\lambda_{2,i}\}_{i=1}^r$ of $\mathcal{C}(G,S_2)$ are real.\;This is a contradiction since Cayley graph $\mathcal{C}(G,S_2)$ is a regular non-empty directed graph without undirected edges,\;it has at least one non-real eigenvalue.
\end{proof}

Then we can get the following corollary easily since any conjugacy class of an abelian group $G$ just consist of only one element.\;

\begin{cor}(\cite{J1})A DSRG cannot be a Cayley graph $\mathcal{C}(G,S)$ of an abelian group.\;
\end{cor}

\section{DSRG cannot be a Cayley graph \texorpdfstring{$\mathcal{C}(G,S)$ with $S=\mathbf{C}\cup T$}{Lg}}
For any subset $M$ of $G$,\;recall the set $N_G(M)=\{g\in G|M^g=gMg^{-1}=M\}$,\;is the normalizer
of $M$ in G.\;Let $T$ be an antisymmetric subset of $N_G(M)\setminus(M\cup\{e\})$.\;

\begin{thm}\label{t-normalizer}
 Let $M$ be a subset of $G$ such that $M$ is closed under taking inverses,\;i.e.,\;$M^{(-1)}=M$,\;and $T$ be an antisymmetric subset of $N_G(M)\setminus(M\cup\{e\})$,\;then Cayley graph $\mathcal{C}(G, M\cup T)$ is not a DSRG.
\end{thm}

\begin{proof}
For any $x,\;y\in G$,\;we define the sets $A_{x,y}=\{(c,t)|(c,t)\in {M}\times T,ct=y x^{-1}\}$ and $B_{x,y}=\{(t',c')|(t',c')\in T\times{M},t'c'=y x^{-1}\}$.\;We assert that $|A_{x,y}|=|B_{x,y}|$.\;We prove this assertion by giving a one-to-one correspondence between the sets $A_{x,y}$ and $B_{x,y}$.\;

Indeed,\;if $(c,t)\in A_{x,y}$ such that $ct=y x^{-1}$,\;then the assumption $t\in T\subseteq N_G(M)$ implies $c'=t^{-1}ct\in {M}$,\;so $y x^{-1}=tc'$ and $(t,c')\in B_{x,y}$.\;Conversely,\;if $(t',c')\in B_{x,y}$ satisfies $t'c'=y x^{-1}$,\;then $c=t'c't'^{-1}\in {M}$.\;Hence the array $(c',t)$ satisfies $ct'=y x^{-1}$ and then $(c,t')\in A_{x,y}$.\;This gives a one-to-one correspondence between the sets $A_{x,y}$ and $B_{x,y}$.\;

Therefore,\;$\mathbf{A}(\mathcal{C}(G,M))$ and $\mathbf{A}(\mathcal{C}(G,T))$ are commutative and $\mathbf{A}(\mathcal{C}(G,M))+\mathbf{A}(\mathcal{C}(G,T))=\mathbf{A}(\mathcal{C}(G,M\cup T))$.\;Then the eigenvalues of $\mathbf{A}(\mathcal{C}(G,M\cup T))$ are the sums of the corresponding eigenvalues of $\mathbf{A}(\mathcal{C}(G,M))$ and $\mathbf{A}(\mathcal{C}(G,T))$.\;

If Cayley graph $\mathcal{C}(G, M\cup T)$ is a DSRG,\;then all the eigenvalues of $\mathbf{A}(\mathcal{C}(G,M\cup T))$ are real numbers.\;Note that all the eigenvalues of $\mathbf{A}(\mathcal{C}(G,M))$ are real since $\mathbf{A}(\mathcal{C}(G,M))$ is a symmetric matrix,\;this shows that all the eigenvalues of $\mathbf{A}(\mathcal{C}(G,T))$  are real.\;But Lemma \ref{l-withoutundirected} implies that at least one eigenvalues of $\mathbf{A}(\mathcal{C}(G,T))$ is not real.\;This is a contradiction.\;
\end{proof}

Let $S=\mathbf{C}\cup T$.\;where $\mathbf{C}$  is a union of some conjugate classes of $G$ and $T$ is an antisymmetric subset such that $\mathbf{C}\cap T=\emptyset$.\;From Theorem \ref{t-normalizer},\;we have the following corollary.

\begin{cor}DSRG cannot be a Cayley graph \texorpdfstring{$\mathcal{C}(G,S)$}{Lg} with $S=\mathbf{C}\cup T$,\;where $\mathbf{C}$ is a union of some conjugate classes of $G$ and $T$ is an antisymmetric subset of $G$ such that $\mathbf{C}\cap T=\emptyset$.\;
\end{cor}
\begin{proof}Observe that the normalizer of $\mathbf{C}$ is $G$.\;As the proof of Theorem \ref{t-conjugacy},\;we can write $\mathbf{C}=\mathbf{C}_1\cup \mathbf{C}_2$,\;where $\mathbf{C}_1=\bigcup_{\{C_i:C_i^{(-1)}\subseteq \mathbf{C}\}}C_i$ and $\mathbf{C}_2=\bigcup_{\{C_i:C_i^{(-1)}\nsubseteq \mathbf{C}\}}C_i$.\;Let $\mathbf{C}'=\mathbf{C}_1$ and $T'=T\cup \mathbf{C}_2\subseteq G=N_G(\mathbf{C}')$,\;then $\mathbf{C}'=\mathbf{C}'^{(-1)}$ is closed under taking inverses and $T'\subseteq G\setminus (\mathbf{C}_1\cup\{e\})$.\;Then this corollary follows from the Theorem \ref{t-normalizer} with
$M=\mathbf{C}'$ and $T=T'$.\;
\end{proof}

\section{A generalization of semidirect product constructions of DSRGs}
Let $G$ be a finite group,\;and $\mathbb{Z}[G]$ be the integral group ring.\;Define
\[\mathbb{Z}_{\geqslant0}[G]=\left\{\sum_{g\in G}a_gg\bigg|a_g\in\mathbb{Z},\;a_g\geqslant 0,\;\forall g\in G\right\}.\]
Let $\mathcal{P}(G)=\{S:S \text{\;is a multisubset of\;} G\}$ be the set of all the multisubsets of $G$.\;Define
\begin{equation}\label{pG}
\mathcal{P}_G:\mathbb{Z}_{\geqslant0}[G]\rightarrow\mathcal{P}(G),\;\sum_{g\in G}a_gg\mapsto \biguplus_{g\in N}a_g\oplus \{g\}.
\end{equation}
This is a one-to-one correspondence between  $\mathbb{Z}_{\geqslant0}[G]$ and $\mathcal{P}(G)$.\;

\subsection{Preliminary}
\begin{dfn}(Semidirect product) Let $N$ and $H$ be two groups.\;Let $\theta:H\rightarrow \mathbf{Aut}(N)$ be a given homomorphism from $H$ to  $\mathbf{Aut}(N)$.\;Let
$N\rtimes_\theta H$ be the direct product set of $N$ and $H$,\;with the following operation for the product of two elements
\begin{equation*}
(n,h)(n',h')=(n[\theta(h)(n')], hh').
\end{equation*}
Then $N\rtimes_\theta H$ is called the semidirect product of $N$ and $H$
with respect to the  homomorphism $\theta$.\;
\end{dfn}

We can define an action of $H$ on $N$.\;Let $N\rtimes_\theta H$ be the semidirect product of $N$ and $H$,\;then $H$ acts on the $N$ by defining $h\circ n:=hnh^{-1}=\theta(h)(n)$.\;Then the orbit of $n$ is the set $H\circ n=\{h\circ n:h\in H\}$ and hence $N$ has a unique partition consisting of orbits.\;There is a trivial $H$-orbit only consists of the identity element $e_N$,\;and the other $H$-orbits are called the untrivial $H$-orbits of $N$.\;

Let $N_1$ be a subset of $N$,\;we define
\begin{equation}\label{musical}
N_1^{\flat}=\mathcal{P}_N\left(\sum_{h\in H}h\overline{N_1}h^{-1}\right),
\end{equation}
where $\mathcal{P}_N$ is defined as ($\ref{pG}$).\;Therefore $N_1^{\flat}$ is a multisubet of $N$,\;and
\[\overline{N_1^{\flat}}=\sum_{h\in H}h\overline{N_1}h^{-1}.\]
The following example gives an interpretation of the above notations.\;
\begin{exmp}Considering the dihedral group $D_8=C_4\rtimes C_2$,\;where $C_4=\langle x|x^4=1\rangle$.\;Then $C_2\circ x^2=\{x^2\}$ and $\{x^2\}^\flat=\mathcal{P}_{C_4}(2x^2)=\{x^2,x^2\}$.\;If $N_1=\{x^1,x^2,x^3\}$,\;then
$$N_1^{\flat}=\mathcal{P}_{C_4}(2x^1+2x^3+2x^2)=\{x^1,x^1,x^2,x^2,x^3,x^3\}.$$
\end{exmp}

\subsection{Semidirect product constructions of DSRGs}
Let $N$ be a finite group of order $m$.\;Art M.\;Duval and Dmitri Iourinski in \cite{D} define a group automorphism $\beta\in \mathbf{Aut}(N)$ has the $q$-orbit condition if each of its untrivial orbits contains $q$ elements.\;In other words,\;$\beta$ satisfies $\beta^{q}(a) = a$,\;for all $a\in N$,\;and $\beta^{u}(a) = a$ implies $q|u$,\;for all $a \neq e_{N}$.\;

We can also define the $q$-orbit condition in terms of group action.\;Let $\beta\in \mathbf{Aut}(N)$ be a group automorphism,\;then the subgroup $\langle\beta\rangle$,\;i.e.,\;the subgroup of $\mathbf{Aut}(N)$ generated by $\beta$,\;acts on $N$ naturally.\;Then the $\beta\in \mathbf{Aut}(N)$ has the $q$-orbit condition provided  each untrivial $\langle\beta\rangle$-orbit have $q$ elements.\;

The following constructions of DSRGs were given by Art M.\;Duval and Dmitri Iourinski.
\begin{thm}(\cite{D})\label{semidirectDSRG}
 Let $N$ be a finite group of order $m$.\;Let $\beta \in \mathbf{Aut}(N)$ have the $q$-orbit condition.\;Let $H$ be the cyclic group of order $q$ with generator $b$,\;and define $\theta: H\rightarrow \mathbf{Aut}(N)$ by $\theta(b^{u})=\beta^{u}$.\;Let $N_1$ be a set of representatives of the nontrivial orbits of $\beta$.\;Then the Cayley graph $\mathcal{C}(N\rtimes_{\theta}H,N_1\times H)$
is a DSRG with the parameters
$$\left(mq,m-1,\frac{m-1}{q}, \frac{m-1}{q}-1,\frac{m-1}{q}\right).$$
\end{thm}
\subsection{Induced representations (modules)}

We now introduce some basic concepts of \emph{induced representations (modules)}.\;Let $H\leqslant G$ and $(V,\rho)$ be an $F$-linear representation of $H$.\;We can  extend it as follows
\[\left(\sum_{h}a_hh\right)v:=\sum_{h}a_h(hv).\]
Then $V$ has an $F[H]$-module structure.\;

Since the group algebra $F[G]$ is an $(F[G],F[H])$-bimodule,\;we can construct the $F[G]$-module $F[G]\otimes_{F[H]}V$.\;We call this $F[G]$-module the \emph{induced} $F[G]$-\emph{module} of $V$ (or the \emph{induction} of $V$ to $G$),\;and we denote it by $\mathbf{Ind}_{H}^GV$.\;We denote the
$F$-linear representation arising from the $\mathbf{Ind}_{H}^GV$ by $\rho^G$,\;and we call $\rho^G$  an \emph{induced representation}.\;

Let $B=\{v_1,\cdots,v_{m}\}$ be an $F$-basis of $V$ and $\{g_1,\cdots,g_d\}$ be a (left) transversal for $H$ in $G$.\;Let $\rho_B(g)=(a_{ij}(g))_{m\times m}$ for any $g\in H$ and
\[\dot{a}_{ij}(g)=\left\{
  \begin{array}{ll}
    a_{ij}(g), & \hbox{$g\in H$,} \\
    0, & \hbox{$g\not\in H$.}
  \end{array}
\right.\]
The following lemma gives some basic propositions about induced representations.\;
\begin{lem}\label{l-InducedModules}
(1)\;$\dim_{F}\mathbf{Ind}_{H}^GV=|G:H|\dim_{F}V$.\\
(2)\;The set $C=\{g_i\otimes v_j|1\leq i\leq d,\;1\leq j\leq m\}=\{g_1\otimes v_1,\cdots,g_1\otimes v_m,\cdots,g_d\otimes v_d,\cdots,g_d\otimes v_m\}$ is an $F$-basis of $\mathbf{Ind}_{H}^GV$.\;Moreover,\;as an $F$-vector space we have
\[\mathbf{Ind}_{H}^GV=\bigoplus_{i=1}^d g_i(1\otimes_{F[H]}V).\]
(3)\;Let $\dot{\rho}_B(g)=(\dot{a}_{ij}(g))_{m\times m}$.\;Then
\begin{equation}\label{inducedReMatrix}
\rho^G_{C}(g)=\left(
    \begin{array}{ccc}
      \dot{\rho}_B(g_1^{-1}gg_1) & \cdots & \dot{\rho}_B(g_1^{-1}gg_d) \\
      \cdots & \cdots & \cdots \\
      \dot{\rho}_B(g_d^{-1}gg_1) & \cdots & \dot{\rho}_B(g_d^{-1}gg_d) \\
    \end{array}
  \right),
\end{equation}
where $\rho_B(g)$ is the matrix of representation $\rho(g)$ with respect to the basis $B$,\;and $\rho^G_{C}(g)$ is the matrix of representation $\rho^G(g)$ with respect to the basis $C$,\;for any $g\in G$.\\
(4)\;Let $F=\mathbb{C}$ and $(V,\rho)$ be a $\mathbb{C}$-linear representation of $H$,\;the $\mathbb{C}$-character of $\rho$ is $\chi$.\;We denote the character of $\rho^G$ by $\chi^G$,\;and we call $\chi^G$ an induced character.\;Then for any $g\in G$ we have
\[\chi^G(g)=\sum_{i=1}^d\dot{\chi}(g_i^{-1}gg_i).\]
(5)\;Let $U$ be an $F[G]$-module,\;then $F[G]$ can be regarded as an $F[H]$-module,\;this $F[H]$-module is called the restriction of $U$ to $H$ and we denote it by $\mathbf{Res}_H^GU$.\;Furthermore,\;if $U$ is an $\mathbb{C}[G]$-module having character $\chi$,\;then we denote the character of the $\mathbb{C}[H]$-module $\mathbf{Res}_H^GU$ by $\chi|_H$.
\end{lem}

\begin{lem}\label{regInduced}Let $\rho_{reg}$ be the left regular representation of $H$.\;Then $\rho_{reg}^G$ is the left regular representation of $G$.
\end{lem}
\begin{proof}Note that the representation space of $\rho_{reg}$ is $\mathbb{C}[H]$.\;Therefore
$$\mathbf{Ind}_{H}^G\mathbb{C}[H]=\mathbb{C}[G]\otimes_{\mathbb{C}[H]}\mathbb{C}[H]\cong \mathbb{C}[G]$$
is a left $\mathbb{C}[G]$-module isomorphism.\;This gives that $\rho_{reg}^G$ is the left regular representation of $G$.\;
\end{proof}

\subsection{Induced representation and directed strongly regular Cayley graph \texorpdfstring{$\mathcal{C}(N\rtimes_\theta H,N_1\times H)$}{Lg}}
Let $G=N\rtimes_\theta H$,\;$|N|=m$ and $|H|=d$,\;then the (left) transversal for $N$ in $G$ is $H=\{h_1,\cdots,h_d\}$.\;

In this subsection,\;we have a look at the Cayley graph $\mathcal{C}(N\rtimes_\theta H,N_1\times H_1)$,\;where $N_1$ and $H_1$ are subsets of $N$ and $H$ respectively.\;

Let $\varrho=\rho_{reg}$ be the left regular representation of $N$,\;then the representation space with respect to $\rho$ is $\mathbb{C}[N]$.\;Thus $N=\{n_1,n_2,\cdots,n_m\}$ is a $\mathbb{C}$-basis of $\mathbb{C}[N]$,\;and $\varrho_N(g)$ is the matrix of representation $\rho(g)$ with respect to the basis $N$,\;for any $g\in G$.\;

The following lemma gives the adjacent matrix of Cayley graph $\mathcal{C}(N\rtimes_\theta H,N_1\times H_1)$.\;

\begin{lem}\label{l-AdjacentMatrix}Let $\mathcal{C}(N\rtimes_\theta H,N_1\times H_1)$ be the Cayley graph defined above.\;Then
\begin{equation}\label{5.4}
\mathbf{A}(\mathcal{C}(N\rtimes_\theta H,N_1\times H_1))\\
=\left(
    \begin{array}{ccc}
      \varrho_N(h_1^{-1}\overline{N_1}h_1) & \cdots & 0 \\
      \vdots & \ddots & \vdots \\
      0 & \cdots & \varrho_N(h_m^{-1}\overline{N_1}h_m) \\
    \end{array}
  \right)\mathbf{A}(\mathcal{C}(H, H_1^{(-1)}))\otimes I_m.
\end{equation}
\end{lem}
\begin{proof}The notations are coincided with above.\;It follows from Lemma \ref{regInduced} that $\varrho^G=\rho_{reg}^G$ is the left regular representation of $G$,\;and $D=\{h_1\otimes n_1,\cdots,h_1\otimes n_m,\cdots,h_d\otimes n_1,\cdots,h_d\otimes n_m\}$ is a $\mathbb{C}$-basis of $\mathbf{Ind}_{H}^G\mathbb{C}[N]\cong\mathbb{C}[G]$.\;Note that the basis $D$ can be identified with the elements of $G$,\;hence $\varrho^G_{D}(\overline{N_1\times H_1})$ is an adjacent matrix of Cayley graph $\mathcal{C}(N\rtimes_\theta H, N_1\times H_1)$.\;Therefore
\begin{equation}\label{5.5}
\mathbf{A}(\mathcal{C}(N\rtimes_\theta H,N_1\times H_1))=\varrho^G_{D}(\overline{N_1\times H_1})=\varrho^G_{D}(\overline{N_1})\varrho^G_{D}(\overline{H_1}).
\end{equation}
For any $h\in H_1$,\;Lemma \ref{l-InducedModules} $(3)$ gives that
\begin{equation}\label{5.6}
\begin{aligned}
\varrho^G_{D}(h)&=\left(
    \begin{array}{ccc}
      \dot{\varrho}_N(h_1^{-1}hh_1) & \cdots & \dot{\varrho}_N(h_1^{-1}hh_d) \\
      \vdots & \ddots & \vdots \\
      \dot{\varrho}_N(h_d^{-1}hh_1) & \cdots & \dot{\varrho}_N(h_d^{-1}hh_d) \\
    \end{array}
  \right)\\
&=\left(\begin{array}{ccc}
      \delta_{11}I_m & \cdots & \delta_{1d}I_m \\
      \vdots & \ddots & \vdots \\
      \delta_{d1}I_m & \cdots & \delta_{dd}I_m \\
    \end{array}
  \right)\\
&=\mathbf{A}(\mathcal{C}(H,\{h^{-1}\}))\otimes I_m,
\end{aligned}
\end{equation}
where
 $$\delta_{ij}=\left\{
  \begin{array}{ll}
    1, & \hbox{if $h_jh_i^{-1}=h^{-1}$;} \\
    0, & \hbox{if $h_jh_i^{-1}\neq h^{-1}$.}
  \end{array}
\right.$$
The second equality follows from the fact that $\dot{\varrho}_N(h_i^{-1}hh_j)\neq 0$ if and only if $h_i^{-1}hh_j\in N\cap H=e$,\;i.e.,\;$h_jh_i^{-1}=h^{-1}$.\;In this case,\;$\dot{\varrho}_N(h_i^{-1}hh_j)={\varrho}_N(e)=I_m$.\;

It follows from the equation (\ref{5.5}) that
\begin{equation}\label{5.7}
\begin{aligned}
\varrho^G_{D}(\overline{H_1})&=\sum_{h\in H_1}\varrho^G_{D}(h)=\sum_{h\in H_1}\mathbf{A}(\mathcal{C}(H,\{h^{-1}\}))\otimes I_m\\
&=\mathbf{A}(\mathcal{C}(H, H_1^{(-1)}))\otimes I_m.
\end{aligned}
\end{equation}
Meanwhile,\;we also have
\begin{equation}\label{5.8}
\begin{aligned}
\varrho^G_{D}(\overline{N_1})&=\left(
    \begin{array}{ccc}
      \dot{\varrho}_N(h_1^{-1}\overline{N_1}h_1) & \cdots & \dot{\varrho}_N(h_1^{-1}\overline{N_1}h_d) \\
      \vdots & \ddots & \vdots \\
      \dot{\varrho}_N(h_d^{-1}\overline{N_1}h_1) & \cdots & \dot{\varrho}_N(h_d^{-1}\overline{N_1}h_d) \\
    \end{array}
  \right)\\
&=\left(
    \begin{array}{ccc}
      \varrho_N(h_1^{-1}\overline{N_1}h_1) & \cdots & 0 \\
      \vdots & \ddots & \vdots \\
      0 & \cdots & \varrho_N(h_m^{-1}\overline{N_1}h_m) \\
    \end{array}
  \right).
\end{aligned}
\end{equation}
The second equality follows from the fact that $\dot{\varrho}_N(h_i^{-1}\overline{N_1}h_j)\neq 0$ implies that $h_i^{-1}\overline{N_1}h_j=h_i^{-1}\overline{N_1}h_i(h_i^{-1}h_j)\subseteq N$,\;then $h_i^{-1}h_j\in N\cap H=e$,\;i.e.\;$h_i=h_j$.\;

Therefore,\;(\ref{5.5}),\;(\ref{5.7}) and (\ref{5.8}) show that
\begin{equation}\label{5.9}
\varrho^G_{D}(\overline{N_1\times H_1})=\left(
    \begin{array}{ccc}
      \varrho_N(h_1^{-1}\overline{N_1}h_1) & \cdots & 0 \\
      \vdots & \ddots & \vdots \\
      0 & \cdots & \varrho_N(h_m^{-1}\overline{N_1}h_m) \\
    \end{array}
  \right)\mathbf{A}(\mathcal{C}(H, H_1^{(-1)}))\otimes I_m.
\end{equation}
This completes the proof.\;
\end{proof}
For $H_1=H$,\;we will determine the characteristic polynomial and  minimal polynomial of Cayley graph $\mathcal{C}(N\rtimes_\theta H,N_1\times H)$,\;by using the Cayley multigraph $\mathcal{C}(N,N_1^\flat)$,\;where $N_1^\flat$ is defined in (\ref{musical}).\;Throughout this section,\;let
\[L=\mathbf{A}(\mathcal{C}(N\rtimes_\theta H, N_1\times H))\text{\;and\;}K=\mathbf{A}(\mathcal{C}(N,N_1^\flat))\]
be the adjacent matrices of the Cayley graph $\mathcal{C}(N\rtimes_\theta H, N_1\times H)$ and the Cayley multigraph $\mathcal{C}(N,N_1^\flat)$.\;

\begin{lem}\label{l-CharacteristicMinimum}The characteristic polynomial of the Cayley graph $\mathcal{C}(N\rtimes_\theta H, N_1\times H)$ is
\[\mathcal{F}_L(x)=x^{md-m}\mathcal{F}_{K}(x).\]
The minimal polynomial of the Cayley graph $\mathcal{C}(N\rtimes_\theta H, N_1\times H)$ satisfies
\[\mathcal{M}_{K}(x)|\mathcal{M}_L(x),\mathcal{M}_L(x)|x\mathcal{M}_{K}(x).\]
\end{lem}
\begin{proof}From Lemma \ref{l-AdjacentMatrix} and the fact that $\mathbf{A}(\mathcal{C}(H, H^{(-1)}))=J_d$,\;the equation $(\ref{5.4})$ becomes
\begin{equation*}
\begin{aligned}
L&=\left(
    \begin{array}{ccc}
      \varrho_N(h_1^{-1}\overline{N_1}h_1) & \cdots & 0 \\
      \vdots & \ddots & \vdots \\
      0 & \cdots & \varrho_N(h_d^{-1}\overline{N_1}h_d) \\
    \end{array}
  \right)J_d\otimes I_m\\
&=\left(
    \begin{array}{ccc}
      \varrho_N(h_1^{-1}\overline{N_1}h_1) & \cdots & \varrho_N(h_1^{-1}\overline{N_1}h_1) \\
      \vdots & \ddots & \vdots \\
      \varrho_N(h_d^{-1}\overline{N_1}h_d) & \cdots & \varrho_N(h_d^{-1}\overline{N_1}h_d) \\
    \end{array}
  \right)\\
&=\left(
     \begin{array}{c}
       \varrho_N(h_1^{-1}\overline{N_1}h_1) \\
       \vdots \\
       \varrho_N(h_d^{-1}\overline{N_1}h_d) \\
     \end{array}
   \right)
\left(
  \begin{array}{ccc}
    I_m & \cdots & I_m \\
  \end{array}
\right).
\end{aligned}
\end{equation*}
Then
\begin{equation*}
\begin{aligned}
\mathcal{F}_{L}(x)&=\left|xI_{md}-L\right|\\
&=x^{md-m}\left|xI_{m}-\left(
  \begin{array}{ccc}
    I_m & \cdots & I_m \\
  \end{array}
\right)\left(
     \begin{array}{c}
       \varrho_N(h_1^{-1}\overline{N_1}h_1) \\
       \vdots \\
       \varrho_N(h_d^{-1}\overline{N_1}h_d) \\
     \end{array}
   \right)
\right|\\
&=x^{md-m}\left|xI_{m}-\sum_{j=1}^d\varrho_N(h_j^{-1}\overline{N_1}h_j)\right|.
\end{aligned}
\end{equation*}
Note that
\[\sum\limits_{j=1}^d\varrho_N(h_j^{-1}\overline{N_1}h_j)=\varrho_N(\overline{N_1^\flat})=\mathbf{A}(\mathcal{C}(N,N_1^\flat))=K\]
is the adjacency matrix of Cayley multigraph $\mathcal{C}(N,N_1^\flat)$,\;so
\begin{equation*}
\mathcal{F}_L(x)=x^{md-m}\left|xI_{m}-\varrho_E(\overline{N_1^\flat})\right|=x^{md-m}\mathcal{F}_{K}(x).
\end{equation*}
Observe that
\begin{equation*}
L^j=\left(
     \begin{array}{c}
       \varrho_N(h_1^{-1}\overline{N_1}h_1) \\
       \vdots \\
       \varrho_N(h_d^{-1}\overline{N_1}h_d) \\
     \end{array}
   \right)K^{j-1}
\left(
  \begin{array}{ccc}
    I_m & \cdots & I_m \\
  \end{array}
\right)
\end{equation*}
for any $j\geqslant 1$.\;Note that $0$ is a root of minimal polynomial $\mathcal{M}_L(x)$,\;since $0$ is a root of characteristic polynomial $\mathcal{F}_L(x)$.\;Therefore we can let $\mathcal{M}_L(x)=xm(x)$ for some ploynomial $m(x)$,\;then
\begin{equation*}
0=\mathcal{M}_L(L)=\left(
     \begin{array}{c}
       \varrho_N(h_1^{-1}\overline{N_1}h_1) \\
       \vdots \\
       \varrho_N(h_d^{-1}\overline{N_1}h_d) \\
     \end{array}
   \right)m(K)
\left(
  \begin{array}{ccc}
    I_m & \cdots & I_m \\
  \end{array}
\right).
\end{equation*}
Multiplying $(
  \begin{array}{ccc}
    I_m & \cdots & I_m \\
  \end{array})$ on both sides of the above equation,\;we can get
\begin{equation*}
0=\left(
  \begin{array}{ccc}
    I_m & \cdots & I_m \\
  \end{array}
\right)\mathcal{M}_L(K)=\left(
  \begin{array}{ccc}
    Km(K) & \cdots & Km(K) \\
  \end{array}
\right).
\end{equation*}
This gives that $Km(K)=\mathcal{M}_L(K)=0$ and hence $$\mathcal{M}_{K}(x)|\mathcal{M}_L(x).\;$$
On the other hand,\;we note that
\begin{equation*}
\begin{aligned}
0&=\left(
     \begin{array}{c}
       \varrho_N(h_1^{-1}\overline{N_1}h_1) \\
       \vdots \\
       \varrho_N(h_d^{-1}\overline{N_1}h_d) \\
     \end{array}
   \right)\mathcal{M}_{K}(K)
\left(
  \begin{array}{ccc}
    I_m & \cdots & I_m \\
  \end{array}
\right)=L\mathcal{M}_{K}(L),
\end{aligned}
\end{equation*}
this gives that $\mathcal{M}_L(x)|x\mathcal{M}_{K}(x)$.\;
\end{proof}

\begin{thm}\label{t-Cayleymultidigraph}
 If Cayley graph $\mathcal{C}(N\rtimes_\theta H, N_1\times H)$ is a DSRG with parameters $(md, k, \mu, \lambda, t)$,\;then Cayley multigraph $\mathcal{C}(N,N_1^\flat)$ satisfies $\overline{N_1^\flat}^2=\sigma\overline{N_1^\flat}+\frac{k(k-\sigma)}{m}\overline{N}$ or $\overline{N_1^\flat}=\frac{k-\sigma}{m}\overline{N}+\sigma e_N$.
\end{thm}

\begin{proof}Suppose Cayley graph $\mathcal{C}(N\rtimes_\theta H, N_1\times H)$ is a DSRG with parameters $(md, k, \mu, \lambda, t)$,\;then this Cayley graph $\mathcal{C}(N\rtimes_\theta H, N_1\times H)$ has $3$ distinct eigenvalues $k$,\;$\rho$ and $\sigma$.\;We can get $\rho=0$ and $\sigma<0$ by Lemma \ref{l-CharacteristicMinimum}.\;Then from Remark \ref{mini},
\begin{equation*}
\mathcal{M}_L(x)=x(x-k)(x-\sigma).
\end{equation*}
It follows from Lemma \ref{l-CharacteristicMinimum} that
\begin{equation*}
\mathcal{M}_{K}(x)=(x-k)(x-\sigma)\;\;\text{or}\;\;x(x-k)(x-\sigma).
\end{equation*}
\textbf{Case 1}.\;If $\mathcal{M}_{K}(x)=(x-k)(x-\sigma)$,\;then
$$(K-kI_m)(K-\sigma I_m)=0.$$ Let $X=K-\sigma I_m$,\;then $X$ is a nonnegative matrix and $KX=kX$.\;Therefore each column of $X$ is an eigenvector corresponding to simple eigenvalue $k$ of $K$ (from the Perron-Frobenius theory,\;can see, e.g.,\;Horn and Johnson\cite{RA}),\;but the eigenspace associated with the eigenvalue $k$ has
dimension one and hence each column of $X$ is a suitable multiple of $\mathbf{1}_{m}$,\;where $I_m$ is all-ones vector.\;Therefore,\;there are some integers $b_1,b_2,\cdots,b_m$ such that $X=(b_1\mathbf{1}_{m}, b_2\mathbf{1}_{m}, \cdots, b_{m}\mathbf{1}_{m})$,\;then $\mathbf{1}_{m}^TX=\mathbf{1}_{m}^T(K-\sigma I_m)=(k-\sigma)\mathbf{1}_{m}^T$.\;This shows that ${m}b_1={m}b_2=\cdots={m}b_m=k-\sigma$ and hence
$$X=K-\sigma I_m=\frac{k-\sigma}{m}J_m,\;\text{i.e.,\;}\overline{N_1^\flat}=\frac{k-\sigma}{m}\overline{N}+\sigma e_N.$$
\textbf{Case 2}.\;If $\mathcal{M}_{K}(x)=x(x-k)(x-\sigma)$,\;then
$$K(K-kI_m)(K-\sigma I_m)=0.$$Let $Y=K(K-\sigma I_m)$,\;then $Y$ is a nonnegative matrix and $KY=kY$.\;Similarly,\;we can also obtain that $Y=K(K-\sigma I_m)=\frac{k(k-\sigma)}{m}J_m$,\;so $$K^2=\sigma K+\frac{k(k-\sigma)}{m}J_m.$$
This gives $\overline{N_1^\flat}^2=\sigma\overline{N_1^\flat}+\frac{k(k-\sigma)}{m}\overline{N}.$
\end{proof}

The above lemma motivates us to get the following constructions of DSRGs,\;which generalize the earlier constructions of Art M. Duval and Dmitri Iourinski in \cite{D}.\;
\begin{con}\label{c-Generalsemidirect}Let $\mathcal{C}(N\rtimes_\theta H, N_1\times H)$ be a Cayley graph such that $$\overline{N_1^\flat}=x_1\overline{N}+x_2 e_N$$ for some integers $x_1,x_2$ with $d>x_1>0$ and $x_2<0$,\;then the
Cayley graph $\mathcal{C}(N\rtimes_\theta H, N_1\times H)$ is a DSRG with parameters $$\left(md, x_1m+x_2, \frac{x_1(x_1m+x_2)}{d}, x_2+ \frac{x_1(x_1m+x_2)}{d},  \frac{x_1(x_1m+x_2)}{d}\right).$$
\end{con}
\begin{rmk}The necessary conditions on the parameters of DSRG assert that $0<x_1<d$ and $x_2<0$.\;
\end{rmk}
\begin{proof}The assumption on $N_1^\flat$ implies that $K=x_1J_m+x_2 I_m$,\;then
\begin{equation*}
\mathcal{M}_{K}(x)=(x-(x_1m+x_2))(x-x_2),
\end{equation*}
where $x_1m+x_2=k=|N_1||H|\neq 0$.\;Recall that $L=\mathbf{A}(\mathcal{C}(N\rtimes_\theta H, N_1\times H))$,\;then $\mathcal{M}_L(x)=x(x-k)(x-x_2)$ by Lemma \ref{l-CharacteristicMinimum}.\;Similar to the proof of $\mathbf{Case}$ $2$ of Theorem \ref{t-Cayleymultidigraph},\;we can obtain that $L(L-x_2I)=\frac{k(k-x_2)}{m}J$,\;then
\begin{equation*}
\begin{aligned}
L^2&=\frac{k(k-x_2)}{md}J+x_2L=\frac{x_1(x_1m+x_2)}{d}J+x_2L\\
&=\frac{x_1(x_1m+x_2)}{d}I+\left(x_2+\frac{x_1(x_1m+x_2)}{d}\right)L+\frac{x_1(x_1m+x_2)}{d}(J-I-L).\\
\end{aligned}
\end{equation*}
Thus Cayley graph $\mathcal{C}(N\rtimes_\theta H, N_1\times H)$ is a DSRG with parameters $$\left(md, x_1m+x_2, \frac{x_1(x_1m+x_2)}{d}, x_2+ \frac{x_1(x_1m+x_2)}{d},  \frac{x_1(x_1m+x_2)}{d}\right).$$
\end{proof}
We now give some applications of Construction \ref{c-Generalsemidirect}.\;
\begin{cor}Let $N$ be a finite group of order $m$.\;Let $\beta \in \mathbf{Aut}(N)$ have the $q$-orbit condition.\;Let $H$ be the cyclic group of order $q$ with generator $b$,\;and define $\theta \colon H \rightarrow \mathbf{Aut}(N)$ by $\theta(b^{u})=\beta^{u}$.\;Let $\mathcal{O}_1,\mathcal{O}_2,\cdots,\mathcal{O}_s$ be all the distinct nontrivial $H$-orbits of $N$.\;

Let $r$ be an integer such that $r<q$,\;and let $\widetilde{\mathcal{O}}_i$ be a subset of $\mathcal{O}_i$ with $|\widetilde{\mathcal{O}}_i|=r$ for $1\leqslant i\leqslant s$.\;Let $N_1=\bigcup_{i=1}^s\widetilde{\mathcal{O}}_i$,\;then the Cayley graph $\mathcal{C}(N \rtimes_{\theta} H,N_1 \times H)$ is a DSRG with the parameters $$\left(mq,(m-1)r,\frac{(m-1)r^2}{q},\frac{(m-1)r^2}{q}-r,\frac{(m-1)r^2}{q}\right).$$
\end{cor}
\begin{proof}We note that
\begin{equation*}
N_1^\flat=\biguplus_{i=1}^s\widetilde{\mathcal{O}}_i^{\flat}=\biguplus_{i=1}^sr\oplus\widetilde{\mathcal{O}}_i=r\oplus(N\setminus \{e_N\}),
\end{equation*}
so $\overline{N_1^\flat}=r\overline{N}-re_N$.\;From Construction \ref{c-Generalsemidirect} with $x_1=r,x_2=-r$,\;$\mathcal{C}(N \rtimes_{\theta} H,N_1 \times H)$ is a DSRG with parameters
$\left(mq,(m-1)r,\frac{(m-1)r^2}{q},\frac{(m-1)r^2}{q}-r,\frac{(m-1)r^2}{q}\right)$.
\end{proof}
\begin{rmk}For $r=1$,\;the parameters become $$\left(mq,m-1,\frac{m-1}{q},\frac{m-1}{q}-1,\frac{m-1}{q}\right),\;$$this is the semidirect product constructions of Art M.\;Duval and Dmitri Iourinski in \cite{D}.\;
\end{rmk}

The following corollary shows that the $q$-obrit condition is not necessary.
\begin{cor}\label{c-5.13}Let $G=N\rtimes_\theta H$ be the semidirect product of $N$ and $H$ with respect to the homomorphism $\theta$.\;Let $\mathcal{O}_1=\{e_N\},\mathcal{O}_2,\cdots,\mathcal{O}_s$ be all the  distinct $H$-orbits of $N$,\;and $d_i=|\mathcal{O}_i|$ for $1\leqslant i\leqslant s$.\;Suppose $d_1\leqslant d_2\leqslant\cdots\leqslant d_s$ and $1<d_2|d_i$ for any $2\leqslant i\leqslant s$.\;

Let $r$ be an integer such that $r<d_2$,\;and let $\widetilde{\mathcal{O}}_i$ be a subset of $\mathcal{O}_i$ with $|\widetilde{\mathcal{O}}_i|=\frac{rd_i}{d_2}<d_i$ for $2\leqslant i\leqslant s$.\;Let $N_1=\bigcup_{i=2}^s\widetilde{\mathcal{O}}_i$.\;Then the Cayley graph $\mathcal{C}(N\rtimes_\theta H,  N_1 \times H)$ is a DSRG with parameters
$$\left(|G|,\frac{r|H|(|N|-1)}{d_2},\frac{r^2|H|(|N|-1)}{d_2^2},\frac{r^2|H|(|N|-1)}{d_2^2}-\frac{r|H|}{d_2},\frac{r^2|H|(|N|-1)}{d_2^2}\right).$$
\end{cor}
\begin{proof}Note that
\begin{equation*}
\begin{aligned}
N_1^\flat&=\biguplus_{i=2}^s\widetilde{\mathcal{O}}_i^\flat=\biguplus_{i=2}^s\left(\frac{|H|}{d_i}\frac{rd_i}{d_2}\right)\oplus\mathcal{O}_i=\frac{r|H|}{d_2}\oplus\left(\biguplus_{i=2}^s\mathcal{O}_i\right)=\frac{r|H|}{d_2}\oplus(N\setminus\{e_N\}).
\end{aligned}
\end{equation*}
Hence $\overline{N_1^\flat}=\frac{r|H|}{d_2}(\overline{N}-e_N)$.\;Then this result follows from Construction \ref{c-Generalsemidirect} with $x_1=\frac{r|H|}{d_2},x_2=-\frac{r|H|}{d_2}$.
\end{proof}
There are some groups $N$,\;$H$ and homomorphism $\theta$ satisfy the conditions of Corollary \ref{c-5.13}.\;

\begin{exmp}If $\beta_i \in \mathbf{Aut}(N_i)$ has the $q_i$-orbit condition for $i=1,\cdots,l$,\;and $q_i|q_{i+1}$ for each $i=1,\cdots,l-1$,\;then
we define
\[\beta=\prod_{j=1}^l\beta_i\in \mathbf{Aut}\left(\prod_{j=1}^lN_j\right).\]
Let $H$ be the cyclic group of order $q_l$ with generator $b$,\;and define $\theta \colon H \rightarrow \mathbf{Aut}\left(\prod_{j=1}^lN_j\right)$ by $\theta(b^{u})=\beta^{u}$.\;Assume $(n_1,n_2,\cdots,n_l)\neq e$ is a representative of an untrivial $H$-orbit $\mathcal{O}$.\;Let $|\mathcal{O}|=r$.\;There is a largest $j$ such that $n_{j}\neq e_{N_{j}}$.\;Therefore $\beta^{q_{j}}(n_1,n_2,\cdots,n_l)=e$ and $r|q_j$.\;On the other hand,\;note that $\beta_j^{r}(n_j)=e_{N_{j}}$,\;it follows that $q_j|r$.\;Therefore $r=q_j$.\;This gives that the length of each untrivial $H$-orbit belongs to the set $\{q_1,q_2,\cdots,q_l\}$.\;Observe that there exist $H$-orbits of size $q_1$ and $q_l$,\;then from Corollary \ref{c-5.13},\;we can construct DSRGs with parameters
$$\left(mq_l,\frac{rq_l(m-1)}{q_1},\frac{r^2q_l(m-1)}{q_1^2},\frac{r^2q_l(m-1)}{q_1^2}-\frac{rq_l}{q_1},\frac{r^2q_l(m-1)}{q_1^2}\right)$$
for any $r<q_1$,\;where $m=\prod_{j=1}^l|N_j|$.\;We give a simple example.\;

Considering the direct product of two cyclic groups $C_3\times C_5$,\;where $C_3=\langle a_1\rangle$ and $C_5=\langle a_2\rangle$.\;Defining $\beta\in \mathbf{Aut}(C_3\times C_5)$ by
$\beta(a_1)=a_1^2$ and $\beta(a_2)=a_2^2$.\;Then $\beta^4$ is the identity automorphism of $C_3\times C_5$.\;Let $C_4$ be the cyclic group of order $4$ with generator $b$,\;and define $\theta: C_4 \rightarrow \mathbf{Aut}(C_3\times C_5)$ by $\theta(b^{u})=\beta^{u}$.\;Then
\[(C_3\times C_5)\rtimes_\theta C_4=\langle a_1,a_2,b:a_1^3=a_2^5=b^4=e,a_1a_2=a_2a_1,ba_1=a_1^2b,ba_2=a_2^2b\rangle.\]
The untrivial $C_4$-orbits are \[\{a_1,a_1^2\},\;\{a_2,a_2^2,a_2^3,a_2^4\},\;\{a_1a_2,a_1^2a_2^2,a_1a_2^4,a_1^2a_2^3\},\; \{a_1a_2^2,a_1^2a_2^4,a_1a_2^3,a_1^2a_2\}.\]
Therefore $C_3\times C_5$,\;$C_4$ and homomorphism $\theta$ satisfy the conditions in Corollary \ref{c-5.13}.\;Let $S=\{a_1,a_2,a_2^2,a_1a_2,a_1^2a_2^2,a_1a_2^2,a_1^2a_2^4\}$,\;then the Cayley graph $\mathcal{C}((C_3\times C_5)\rtimes_\theta C_4, S\times C_4)$ is a DSRG with parameters
$(60,28,14,12,14)$.\;
\end{exmp}

\begin{exmp}
Before describing another example,\;we need a lemma and we can see it in \cite{IS},\;p.71,\;Corollary 3.4.\;\\
$\mathbf{{Lemma}.}$\;\;Let $P$ be an abelian $p$-subgroup of $\mathbf{Aut}(G)$,\;where $G$ is a finite group of order not divisible by the prime $p$.\;Then $P$ has a regular orbit (i.e.\;a orbit of length $|P|$) on $G$.

Considering $N=C_{133}=\langle x|x^{133}=1\rangle$,\;the automorphism group of $N=C_{133}$ is isomorphic to $\mathbb{Z}_{133}^\ast$.\;i.e.,\;
\begin{equation*}
\mathbf{Aut}(C_{133})=\{\gamma_{s}|(s,133)=1,0\leqslant s\leqslant 132\},|\mathbf{Aut}(C_{133})|=108,
\end{equation*}
where $\gamma_{s}$ is defined by $\gamma_{s}(x^i)=x^{si}$.\;$\mathbf{Aut}(C_{133})$ has an abelian Sylow $3$-subgroup
\begin{equation*}
\begin{aligned}
H=&\{\gamma_{s}|s=1,4,9,11,16,23,25,30,36,39,43,44,58,64,74,81,85,92,93,99,\\
  &100,102,106,120,121,123,130\}
\end{aligned}
\end{equation*}
of order $27$.\;Define the semidirect product of $C_{133}$ and $H$ by $C_{133}\rtimes H=\{x^i\gamma|x^i\in C_{133},\gamma\in H\}$,\;where  multiplication is defined by $(x^i\gamma)(x^j\gamma')=x^i\gamma(x^j)\gamma\gamma'$.

Let $\mathcal{O}_1=\{e_N\},\mathcal{O}_2,\cdots,\mathcal{O}_s$ be all the distinct $H$-orbits.\;From definition of $H$,\;the $H$-orbit of length $1$ is only $\mathcal{O}_1=\{e_N\}$ (Consider elements $s=92$ and $s=93$).\;From the above Lemma,\;there is an $H$-orbit of length $27$ at least,\;but $27\nmid 133-1=132$,\;$9\nmid 133-1=132$ and $3|132$,\;so there is an $H$-orbit of length $3$ at least,\;then $d_2=3$.\;From Corollary \ref{c-5.13},\;we can construct the DSRGs with parameters $(3591,1188r,396r^2,396r^2-9r,396r^2)$ for $r=1,2$.
\end{exmp}
\section{Directed strongly regular Cayley graphs on dihedral groups}

\subsection{Preliminary}
We recall $C_n=\langle x\rangle$ is a cyclic multiplicative group of order $n$.\;Dihedral groups $D_n$ is a semidirect product of two cyclic groups $C_n=\langle x\rangle$ of order $n$ and $C_2=\langle a\rangle$ of order $2$.\;Any subset $S$ of $D_n$ can be written by the form $S=X\cup Ya$ with $X,Y\subseteq C_n$.\;

We now give a characterization of the Cayley graph $\mathcal{C}(D_n,X\cup Ya)$ to be directed strongly regular.\;

\begin{lem}\label{l-generaldihedrantcharacterization}The Cayley graph $\mathcal{C}(D_n,X\cup Ya)$ is a DSRG with parameters $(2n,|X|+|Y|, \mu, \lambda, t)$ if and only if $X$ and $Y$ satisfy the following conditions:
\begin{flalign}
(i)\;&(\overline{X}+\overline{X^{(-1)}})\overline{Y}=(\lambda-\mu)\overline{Y}+\mu\overline{C_n};\hspace{200pt}\label{6.1}\\
(ii)\;&\overline{X}^2+\overline{Y}\;\overline{Y^{(-1)}}=(t-\mu)e+(\lambda-\mu)\overline{X}+\mu\overline{C_n}.\label{6.2}
\end{flalign}
\end{lem}
\begin{proof}Note that
\begin{equation*}
\begin{aligned}
(\overline{X}+\overline{Ya})^2&=\overline{X}\;\overline{X}+\overline{X}\;\overline{Ya}+\overline{Ya}\;\overline{X}+\overline{Ya}\;\overline{Ya}=\overline{X}\;\overline{X}+\overline{Y}\;\overline{Y^{(-1)}}+(\overline{X}\;\overline{Y}+\overline{Y}\;\overline{X^{(-1)}})a\\
&=\overline{X}^2+\overline{Y}\;\overline{Y^{(-1)}}+\overline{Y}(\overline{X}+\overline{X^{(-1)}})a.
\end{aligned}
\end{equation*}
Thus,\;from Lemma \ref{l-DSRGCayleyGraphsGruopRing},\;the Cayley graph $\mathcal{C}(D_n,X\cup Ya)$ is a DSRG with parameters $(2n,|X|+|Y|, \mu, \lambda, t)$ if and only if
\begin{equation*}
\begin{aligned}
&\overline{X}^2+\overline{Y}\;\overline{Y^{(-1)}}+\overline{Y}(\overline{X}+\overline{X^{(-1)}})a=te+\lambda(\overline{X}+\overline{Ya})+\mu(\overline{D_n}-(\overline{X}+\overline{Ya})-e)\\
=&(t-\mu)e+(\lambda-\mu)\overline{X}+\mu\overline{C_n}+((\lambda-\mu)\overline{Y}+\mu\overline{C_n})a\\
\end{aligned}
\end{equation*}
This  is equivalent to conditions $(\ref{6.1})$ and $(\ref{6.2})$.\;
\end{proof}
When $Y=X$,\;we have the following:
\begin{lem}\label{l-specialdihedrantcharacterization}The Cayley graph $\mathcal{C}(D_n,X\cup Xa)$ is a DSRG with parameters $(2n,2|X|, \mu, \lambda, t)$ if and only if $t=\mu$ and
\begin{equation}\label{6.3}
(\overline{X}+\overline{X^{(-1)}})\overline{X}=(\lambda-\mu)\overline{X}+\mu\overline{C_n}.
\end{equation}
\end{lem}
\begin{rmk}(\ref{6.3}) also implies that
\begin{equation*}
(\overline{X}+\overline{X^{(-1)}})\overline{X^{(-1)}}=(\lambda-\mu)\overline{X^{(-1)}}+\mu\overline{C_n}.
\end{equation*}
Therefore,\;the sum of the above equation and (\ref{6.3}) gives that
\begin{equation}\label{6.4}
(\overline{X}+\overline{X^{(-1)}})^2=(\lambda-\mu)(\overline{X}+\overline{X^{(-1)}})+2\mu\overline{C_n}.
\end{equation}

\end{rmk}
\subsection{Some constructions of directed strongly regular Cayley graphs on dihedral groups}
In the following constructions,\;$v$ is a positive divisor of $n$ and $l=\frac{n}{v}$.\;
\begin{con}\label{c-6.4}
Let $v$ be an odd positive divisor of $n$.\;Let $T$ be a subset of $\{x^1,\cdots,x^{v-1}\}$,\\and $X$ be a subset of ${C}_n$ satisfy the following conditions:\\
$(i)$\;$X=T\langle x^v\rangle$.\\
$(ii)$\;$X\cup X^{(-1)}=C_{n}\setminus \langle x^v\rangle$.\\
Then the Cayley graph $\mathcal{C}(D_n,X\cup Xa)$ is a DSRG with parameters $\left(2n,n-l,\frac{n-l}{2},\frac{n-l}{2}-l,\frac{n-l}{2}\right)$.\;
\end{con}
\begin{proof}Note that $\overline{X}(\overline{X}+\overline{X^{(-1)}})=-l\overline{X}+\frac{n-l}{2}\overline{C_n}$.\;The result follows from Lemma \ref{l-specialdihedrantcharacterization} directly.\;
\end{proof}
\begin{rmk}
The directed strongly regular Cayley graphs $\mathcal{C}(D_n,X\cup aX)$ constructed above satisfy $X\cap X^{(-1)}=\emptyset$.
\end{rmk}

\begin{con}\label{c-6.6}Let $v>2$ be an even positive divisor of $n$.\;Let $T$ be a subset of $\{x^1,\cdots,x^{v-1}\}$,\\and $X$ be a subset of ${C}_n$ satisfy the following conditions:\\
$(i)$\;$X=T\langle x^v\rangle$.\\
$(ii)$\;$X\cup X^{(-1)}=({C}_{n}\setminus \langle x^v\rangle)\uplus(x^{\frac{v}{2}}\langle x^v\rangle)$.\\
$(iii)$\;$X\cup (x^{\frac{v}{2}}X)={C}_{n}$.\\
Then the Cayley graph $\mathcal{C}(D_n,X\cup Xa)$ is a DSRG with parameters $\left(2n,n,\frac{n}{2}+l,\frac{n}{2}-l,\frac{n}{2}+l\right)$.\;
\end{con}

\begin{proof}Note that $|X|=\frac{n}{2}$ and $\overline{X}+\overline{X^{(-1)}}=\overline{C_n}-\overline{\langle x^v\rangle}+\overline{x^{\frac{v}{2}}\langle x^v\rangle}$.\;Thus $\overline{X}(\overline{X}+\overline{X^{(-1)}})=-l\overline{X}+\frac{n}{2}\overline{C_n}+\overline{T\langle x^v\rangle}\;\overline{x^{\frac{v}{2}}\langle x^v\rangle}=-l\overline{X}+
\frac{n}{2}\overline{C_n}+l\overline{x^{\frac{v}{2}}X}=-l\overline{X}+\frac{n}{2}\overline{C_n}+l\overline{C_n}-l\overline{X}=(\frac{n}{2}+l)\overline{C_n}-2l\overline{X}$.\;The result follows from Lemma \ref{l-specialdihedrantcharacterization} directly.\;
\end{proof}
\begin{rmk}
The directed strongly regular Cayley graphs $\mathcal{C}(D_n,X\cup aX)$ constructed above satisfy $X\cap X^{(-1)}\neq\emptyset$.
\end{rmk}

\begin{con}\label{c-6.8}Let $v$ be an odd positive divisor of $n$.\;Let $T$ be a subset of $ \{e,x^1,\cdots,x^{v-1}\}$ with $e\in T$,\;and let $X,Y\subseteq{C}_n$ satisfy the following conditions:\\
$(i)$\;$Y=T\langle x^v\rangle=X\cup \langle x^v\rangle$.\\
$(ii)$\;$Y\cup Y^{(-1)}=C_{n}\uplus \langle x^v\rangle$.\\
Then the Cayley graph $\mathcal{C}(D_n,X\cup Ya)$ is a DSRG with parameters $\left(2n,n,\frac{n+l}{2},\frac{n-l}{2},\frac{n+l}{2}\right)$.\;
\end{con}
\begin{proof}We have $|Y|=|X|+l=\frac{n+l}{2}$,\;$\overline{X}+\overline{X^{(-1)}}=\overline{C_n}-\overline{\langle x^v\rangle}$ and $\overline{Y}(\overline{X}+\overline{X^{(-1)}})=-l\overline{Y}+\frac{n+l}{2}\overline{C_n}$.\;Meanwhile,\;$\overline{X}^2+\overline{Y}\;\overline{Y^{(-1)}}=\overline{X}(\overline{X}+\overline{X^{(-1)}})+\overline{\langle x^v\rangle}(\overline{X}+\overline{X^{(-1)}})+l\overline{\langle x^v\rangle}=-l\overline{X}+\frac{n+l}{2}\overline{C_n}.$\;
The result follows from Lemma \ref{l-generaldihedrantcharacterization} directly.\;
\end{proof}

\begin{rmk}
It is known that the automorphism group of dihedral group $D_n$ is
\begin{equation*}
\mathbf{Aut}(D_n)=\{\gamma_{s,s'}|0\leqslant s,s'\leqslant n-1,(s,n)=1\},
\end{equation*}
where $\gamma_{s,s'}$ is defined by $\gamma_{s,s'}(x)=x^s,\gamma_{s,s'}(a)=x^{s'}a$.

Let $\alpha\in \mathbf{Aut}(D_n)$,\;then it is easy to see that $\alpha$ is a graph isomorphism from $\mathcal{C}(D_n,X\cup Ya)$ to $\mathcal{C}(D_n,\alpha(X)\cup \alpha(Ya))$.\;Furthermore,\;this asserts that the Cayley graphs $\mathcal{C}(D_n,X\cup Ya)$ and $\mathcal{C}(D_n,X^{(s)}\cup Y^{(s)}x^{s'}a)$ are isomorphic for any $0\leqslant s,s'\leqslant n-1$ and $(s,n)=1$.\;
\end{rmk}
\subsection{The characterization of directed strongly regular Cayley graphs $\mathcal{C}(D_n,X\cup Xa)$ with $X\cap X^{(-1)}=\emptyset$}
Let $\zeta_n$ be a fixed primitive $n$-th root of unity,\;then $\mathbf{Irr}({C_n})=\{\chi_j|0\leqslant j\leqslant n-1\}$,\;where $\chi_j$ is defined by $\chi_j(x^i)=\zeta_n^{ij}$ for $0\leqslant i,j\leqslant n-1$.\;The character $\chi_0$ is called the principal character of $C_n$ and the characters $\chi_1,\chi_2,\cdots,\chi_{n-1}$
are called the nonprincipal characters of $C_n$.\;

Let $X$ be an antisymmetric subset of $C_n$,\;i.e.,\;$X\cap X^{(-1)}=\emptyset$.\;We now give the characterization of the directed strongly regular Cayley graphs $\mathcal{C}(D_n,X\cup Ya)$ with $X\cap X^{(-1)}=\emptyset$.\;
\begin{thm}A Cayley graph $\mathcal{C}(D_n,X\cup Xa)$ with $X\cap X^{(-1)}=\emptyset$ is a DSRG with parameters $\left(2n,2|X|,\mu,\lambda,t\right)$ if and only if there exists a subset $T$ of $\{x^1,\cdots,x^{v-1}\}$ satisfies the following conditions:\\
$(i)$\;$X=T\langle x^v\rangle$;\\
$(ii)$\;$X\cup X^{(-1)}=C_{n}\setminus \langle x^v\rangle$,
where $v=\frac{n}{\mu-\lambda}$.\;
\end{thm}
\begin{proof}Let $U=X\cup X^{(-1)}$,\;then $U$ is a subset of $C_n$ and $\Delta_U(g)\in\{0,1\}$ for any $g\in U$.\;Hence $\overline{U}=\overline{X}+\overline{X^{(-1)}}$.\;It follows from Construction \ref{c-6.4} that a Cayley graph  $\mathcal{C}(D_n,X\cup Xa)$ which satisfies the  conditions  $(i)$ and $(ii)$ is a DSRG with $X\cap X^{(-1)}=\emptyset$.\;

Conversely,\;suppose that $\mathcal{C}(D_n,X\cup Xa)$ is a DSRG with parameters $(2n,2|X|,\mu,\lambda,t)$.\;It follows from
Lemma \ref{l-specialdihedrantcharacterization} and  $(\ref{6.4})$ that
\begin{equation}\label{6.5}
\overline{U}^2=(\lambda-\mu)\overline{U}+2\mu\overline{C_n}.
\end{equation}
Therefore,\;$\chi(\overline{U})\in\{0,\lambda-\mu\}$ for any nonprincipal characters $\chi\in \mathbf{Irr}(C_n)$.\;Define the set
$$\mathcal{U}=\{j:1\leqslant j\leqslant n-1,\;\chi_j(\overline{U})=\lambda-\mu\}.\;$$
Then from the inversion formula (\ref{l-Inversionformula}),\;we have
\begin{equation}\label{6.6}
\Delta_U(g)=\frac{1}{n}\sum_{\chi\in\mathbf{Irr}(C_n)}\chi(\overline{U})\overline{\chi(g)}=\frac{\lambda-\mu}{n}\sum_{j\in \mathcal{U}}\overline{\chi_j(g)}+\frac{2|X|}{n}.
\end{equation}
Note that $e\not\in U$,\;hence $\Delta_U(e)=0$ and then
\[\Delta_U(e)=\frac{\lambda-\mu}{n}|\mathcal{U}|+\frac{2|X|}{n}=0.\]
This gives that that $2|X|=(\mu-\lambda)|\mathcal{U}|$.\;Then (\ref{6.6}) becomes
\begin{equation}\label{6.7}
\Delta_U(g)=\frac{\mu-\lambda}{n}\left(|\mathcal{U}|-\sum_{j\in \mathcal{U}}\overline{\chi_j(g)}\right).
\end{equation}

We assert that $\frac{n}{\mu-\lambda}$ is an integer.\;Indeed,\;selecting some $g\in U$,\;the above equation shows that $|\mathcal{U}|-\sum_{j\in \mathcal{U}}\overline{\chi_j(g)}=\frac{n}{\mu-\lambda}\in \mathbb{Q}$.\;Note that $|\mathcal{U}|-\sum_{j\in \mathcal{U}}\overline{\chi_j(g)}$ also lies in the ring $\mathbb{Z}[\zeta_n]$,\;which is the ring of integers of the $n$-th {cyclotomic field} $\mathbb{Q}(\zeta_n)$.\;Therefore
\[\frac{n}{\mu-\lambda}\in \mathbb{Q}\cap\mathbb{Z}[\zeta_n]=\mathbb{Z}.\]
The equation (\ref{6.7}) also implies that
\begin{equation*}
\begin{aligned}
\Delta_U(g)=0 &\Leftrightarrow |\mathcal{U}|-\sum_{j\in \mathcal{U}}\overline{\chi_j(g)}=0\Leftrightarrow \chi_j(g)=1 \text{\;for\;} j\in\mathcal{U}\Leftrightarrow g\in\bigcap_{j\in\mathcal{U}}\mathcal{K}_{\chi_j}\overset{\text{def}}=R,\\
\end{aligned}
\end{equation*}
where $R$ is some subgroup of $C_n$.\;This shows that $\overline{U}=\overline{C_n}-\overline{R}$.\;Note that $$\overline{U}^2=(\overline{C_n}-\overline{R})^2=(n-2|R|)\overline{C_n}+|R|\overline{R}=(n-|R|)\overline{C_n}-|R|\overline{U},\;$$
so (\ref{6.5}) gives that $|R|=\mu-\lambda$,\;$n-|R|=2\mu$ and $|X|=\frac{n-|R|}{2}=\mu$.\;Then $R=\langle x^{\frac{n}{\mu-\lambda}}\rangle=\langle x^{v}\rangle$,\;proving $(ii)$.\;In this case,\;from Lemma \ref{l-specialdihedrantcharacterization},\;$(\ref{6.3})$ becomes
$$(\lambda-\mu)\overline{X}+\mu\overline{C_n}=\overline{X}\;\overline{U}=\overline{X}(\overline{C_n}-\overline{\langle x^{v}\rangle})=\mu\overline{C_n}-\overline{X}\;\overline{\langle x^{v}\rangle},$$
\;i.e.,\;$(\mu-\lambda)\overline{X}=\overline{X}\;\overline{\langle x^{v}\rangle}$.\;This asserts that $X$ is a union of some cosets of $\langle x^{v}\rangle$ in $C_n$,\;therefore $X=T\langle x^v\rangle$ for some subset $T$ of $\{x^1,\cdots,x^{v-1}\}$,\;proving $(i)$.\;The result follows.\;
\end{proof}

\section*{Acknowledgements}
The authors would like to express their grateful thankfulness to the  referees for their careful reading of the manuscript of the presented paper and valuable comments.
\section*{References}

\bibliographystyle{plain}

\bibliography{1}

\begin{thebibliography}{10}

\bibitem{JL}
Jonathan~L Alperin and Rowen~B Bell.
\newblock {\em Groups and representations}, volume 162.
\newblock Springer Science \& Business Media, 2012.

\bibitem{ST}
Steinberg B.
\newblock {\em Representation Theory of Finite Groups: An Introductory
  Approach}, volume~1.
\newblock Springer-Verlag, 2012.

\bibitem{L}
L$\mathrm{\acute{a}}$szl$\mathrm{\acute{o}}$ Babai.
\newblock Spectra of cayley graphs.
\newblock {\em Journal of Combinatorial Theory, Series B}, 27(2):180--189,
  1979.

\bibitem{A}
Art~M Duval.
\newblock A directed graph version of strongly regular graphs.
\newblock {\em Journal of Combinatorial Theory, Series A}, 47(1):71--100, 1988.

\bibitem{D}
Art~M Duval and Dmitri Iourinski.
\newblock Semidirect product constructions of directed strongly regular graphs.
\newblock {\em Journal of Combinatorial Theory, Series A}, 104(1):157--167,
  2003.

\bibitem{RA}
Roger~A Horn and Charles~R Johnson.
\newblock {\em Matrix analysis}.
\newblock Cambridge university press, 2012.

\bibitem{IS}
I~Martin Isaacs.
\newblock {\em Finite group theory}, volume~92.
\newblock American Mathematical Soc., 2008.

\bibitem{J1}
Leif~K J{\o}rgensen.
\newblock Directed strongly regular graphs with $\mu=\lambda$.
\newblock {\em Discrete Mathematics}, 231(1):289--293, 2001.

\bibitem{K1}
Mikhail Klin, Akihiro Munemasa, Mikhail Muzychuk, and Paul~Hermann Zieschang.
\newblock Directed strongly regular graphs obtained from coherent algebras.
\newblock {\em Linear Algebra and its Applications}, 377(1):83--109, 2004.

\bibitem{JP}
Jean-Pierre Serre.
\newblock {\em Linear representations of finite groups}, volume~42.
\newblock Springer Science \& Business Media, 2012.

\end{thebibliography}
\end{document}